\documentclass[12pt, twoside]{article}
\usepackage{amsfonts}
\usepackage{amsmath,amssymb}
\usepackage{multicol}
\usepackage{color}

\newtheorem{theorem}{Theorem}[section]
\newtheorem{lemma}[theorem]{Lemma}

\newtheorem{remark}[theorem]{Remark}

\numberwithin{equation}{section}
\newenvironment{proof}[1][Proof]{\noindent\textbf{#1.} }{\hfill $\Box$}
\allowdisplaybreaks

 \makeatletter
\begin{document}
\author{\small Yunbo Wang$^a$,\ \  Xiaoyu Zeng$^b$,\ \  Huan-Song Zhou$^b$ $\thanks{Corresponding
author.}$
 \thanks{This work is supported by NSFC (Grant Nos. 11931012, 11871387 and 11871395), and partically supported  by NSF of Hubei Province (2019CFB562).
  \newline \small Email: \texttt{wangyunbo@stumail.nwu.edu.cn}(Y.B. Wang); \texttt{xyzeng@whut.edu.cn}(X.Y. Zeng); \texttt{hszhou@whut.edu.cn}(H.-S. Zhou).}\\
\small  $^a$School of Mathematics and Center for Nonlinear Studies
\\
     \small Northwest University Xi'an 710127, China\\
 \small $^b$Center for Mathematical Sciences, Wuhan University of Technology, Wuhan 430070, China
}
\title{\textbf{\large Asymptotic behavior of least energy solutions for a fractional Laplacian eigenvalue problem on $\mathbb{R}^{N}$
  }}
  \date{}
  \maketitle
\begin{abstract}

  We are interested in the existence and asymptotical behavior for the least energy solutions of the following fractional eigenvalue problem
  \begin{equation*}
  (P)\quad (-\Delta)^{s}u+V(x)u=\mu u+am(x)|u|^{\frac{4s}{N}}u,\quad \int_{\mathbb{R}^{N}}|u|^{2}dx=1,\ u\in H^{s}(\mathbb{R}^{N}),
  \end{equation*}
  where $s\in(0,1)$, $\mu\in\mathbb{R}$, $a>0$, $V(x)$ and $m(x)$ are $L^{\infty}(\mathbb{R}^{N})$ functions with $N\geq2$. We prove that there is a threshold $a_s^{*}>0$ such  that problem $(P)$ has a least energy solution $u_{a}(x)$ for each $a\in(0,a_s^{*})$ and $u_{a}$ blows up, as $a\nearrow a_s^{*}$, at some point $x_0 \in \mathbb{R}^N$,  which makes $V(x_0)$ be the minimum and $m(x_0)$ be the maximum. Moreover, the precise blowup rates for $u_a$  are obtained under suitable conditions on $V(x)$ and $m(x)$.
  \\ \hspace*{\fill} 
\end{abstract}

\section{Introduction}
\noindent

In this paper, we consider the following fractional Laplacian eigenvalue problem
\begin{equation}\label{eq1.1}
  (-\Delta)^{s}u+V(x)u={ \mu} u+am(x)|u|^{p}u,\quad  u\in H^{s}(\mathbb{R}^{N})
\end{equation}
where $s\in(0,1)$, ${ \mu}\in\mathbb{R}$ and $a>0$ are parameters, $N\geq2$,  $p=\frac{4s}{N}$ is the so-called  mass critical exponent, $V(x)$ and $m(x)$ are in $L^{\infty}(\mathbb{R}^{N})$. The fractional Laplace operator $(-\Delta)^{s}$ can be defined as follows
  \begin{equation*}
    (-\Delta)^{s}f(x)=C_{N,s}P.V.\int_{\mathbb{R}^{N}}\frac{f(x)-f(y)}{|x-y|^{N+2s}}dy
    =C_{N,s}\lim_{\epsilon\rightarrow0}\int_{\mathbb{R}^{N}\backslash B_{\epsilon}(0)}\frac{f(x)-f(y)}{|x-y|^{N+2s}}dy,
  \end{equation*}
for any function $f(x)$  in the Schwartz space $\mathcal{S}$.
It is known that $(-\Delta)^{s}$ gives the standard Laplace operator $-\Delta$ if $s=1$, see,  e.g., \cite[Proposition 4.4]{DNE}.

Problem \eqref{eq1.1} arises in studying the standing wave solutions for the following fractional Schr\"{o}dinger equation
\begin{equation}\label{eq1.2}
  -i\phi_{t}+(-\Delta)^{s}\phi+V(x)\phi=am(x)|\phi|^{p}\phi, \ (t,x)\in\mathbb{R}^+\times\mathbb{R}^{N},
\end{equation}
the background for this equation we refer the reader to the papers \cite{NL3,NL4}, etc.

It is easy to see that to seek a standing wave solution for \eqref{eq1.2}, that is, a solution of the form $\phi(t,x)=u(x)e^{-i\mu t}$, we need only to get a solution $u(x)$ of \eqref{eq1.1} in a suitable function space. When $s=1$ and $N=2$, the solutions with prescribed $L^{2}$-norm (or, saying $L^2$-normalized solutions)  for \eqref{eq1.1} have attracted much attention in recent years, since in this case the problem is essentially the Gross-Pitaevskii (GP) equation (with mass critical power $p=\frac{4s}{N}=2$), which is a very important model in the study of {Bose-Einstein condensate} (BEC) theory. For examples, in this case, Bao-Cai \cite{BWZ} and Guo-Seiringer \cite{GYJ1}  studied \eqref{eq1.1} under $m(x)\equiv 1$ and $V(x)$ satisfies
\begin{equation}\label{eq1.3}
  0\leq V(x)\in L^{\infty}_{loc}(\mathbb{R}^{2})\ \text{and}\ \lim_{|x|\rightarrow +\infty}V(x)=+\infty,
\end{equation}
i.e., $V(x)$ is a coercive potential which leads to very  useful compactness in a weighted Sobolev space. By using constrained variational method and energy estimates, they proved that there is a threshold value $a^{*}$, which can be given by the unique positive solution $Q(x)$ (up to translations) of the following equation
\begin{equation*}
  -\Delta u + u = u^{3},\ u\in H^{1}(\mathbb{R}^{2}),
\end{equation*}
i.e., $a^{*}=\|Q\|^{2}_{2}$, such that \eqref{eq1.1} in the above mentioned case has always a least energy solution $u_{a}$ for any $a\in [0,a^*)$, see also \cite{ZJ} for $V(x)=|x|^{2}$. In \cite{GYJ1}, the authors also made some detailed analysis on the asymptotic behavior of the least energy solution $u_{a}$ as $a\nearrow a^{*}$, by assuming the potential $V(x)$ being a polynomial type function with finite global minimum. The results of \cite{GYJ1} have been extended to more general potentials, such as, $V(x)$ is allowed to have infinite many global minimum in \cite{GYJ2,GYJ3},  and $V(x)$ can be a periodic function in \cite{ZD}, etc. Moreover, under the condition \eqref{eq1.3} some results were also obtained for \eqref{eq1.1} in \cite{LL} with $s=1$ and $m(x)\not\equiv 1$ on $\mathbb{R}^{2}$.

Very recently, He-Long \cite{LW} investigated the existence and concentrating behavior of the least energy solutions with given $L^{2}$-norm for \eqref{eq1.1} under $s\in(0,1)$, $m(x)\equiv 1$ and $V(x)$ being bounded. Luo-Zhang \cite{LHJ} studied the existence and nonexistence of normalized solutions for \eqref{eq1.1} with $V(x)\equiv 0$ and combined nonlinearities, i.e., $\mu|u|^{q-2}u+|u|^{p-2}u$, $2<q<p<2^{*}_{s}=\frac{2N}{N-2s}$. Motivated by the mentioned works, it seems natural to ask how about the least energy solutions with prescribed $L^{2}$- norm for \eqref{eq1.1} in more general case: $s\in (0,1)$, $V(x)$ being a bounded function and $m(x)\not\equiv 1$? The aim of this paper is to give a detailed answer to this question. However, different from the above mentioned papers \cite{GYJ1,LW,ZD}, in our case we should be confront with a series difficulties simultaneously, such as the fractional Laplace operator $(-\Delta)^{s}$ is nonlocal which makes the energy estimates more difficult, the potential $m(x)\not\equiv 1$ is involved in the nonlinear term which makes the concentration analysis more  complicated, and the potential $V(x)$ is bounded (not coercive) which leads to the lack of compactness for Sobolev embedding in $\mathbb{R}^N$, etc.

To start our discussions on the equation \eqref{eq1.1}, we first introduce the following fractional Sobolev space,  $s\in (0,1)$,
\begin{equation*}
  H^{s}(\mathbb{R}^{N}):=\Big\{u\in L^{2}(\mathbb{R}^{N}):\frac{|u(x)-u(y)|}{|x-y|^{\frac{N}{2}+s}}\in L^{2}(\mathbb{R}^{N}\times\mathbb{R}^{N})\Big\},
\end{equation*}
with the norm
\begin{equation*}
  \|u\|_{H^{s}(\mathbb{R}^{N})}:=\Big(\int_{\mathbb{R}^{N}}|u|^{2}dx
  +\int_{\mathbb{R}^{N}}\int_{\mathbb{R}^{N}}\frac{|u(x)-u(y)|^{2}}{|x-y|^{N+2s}}dxdy\Big)^{\frac{1}{2}}.
\end{equation*}
Then, $(-\Delta)^{s}$ is well defined in $H^{s}(\mathbb{R}^{N})$. Moreover,  by \cite[Proposition 3.6]{DNE},  we have
\begin{equation*}
\|(-\Delta)^{\frac{s}{2}}u\|^{2}_{2}=\int_{\mathbb{R}^{N}}\int_{\mathbb{R}^{N}}\frac{|u(x)-u(y)|^{2}}{|x-y|^{N+2s}}dxdy.
\end{equation*}

As in the above mentioned papers, we are also concerned with the least energy solutions of \eqref{eq1.1} with given $L^{2}$-norm. So, the Euler-Lagrange principle implies that we may get this kind of solutions for \eqref{eq1.1} by searching the minimizers of the following constrained minimization problem with $\lambda >0$:
\begin{equation}\label{eq1.4}
  I_{\lambda}(a)=\inf\big\{J_{a}(u):\ u\in S_{\lambda}\big\} \text{ and }  I_1(a) = I_\lambda (a) |_{\lambda=1}
\end{equation}
where
\begin{equation*}
  S_\lambda:=\big\{ u\in H^{s}(\mathbb{R}^{N}): \int_{\mathbb{R}^{N}}|u|^{2}dx=\lambda\big\} \text{ and }  S_1= S_\lambda |_{\lambda=1}
\end{equation*}
\begin{equation*}
  J_{a}(u):=\int_{\mathbb{R}^{N}}\big(|(-\Delta)^{\frac{s}{2}}u|^{2}+V(x)|u|^{2}\big)dx
  -\frac{a N}{N+2s}\int_{\mathbb{R}^{N}}m(x)|u|^{\frac{4s}{N}+2}dx,
\end{equation*}
and $J_{a}(u)$ is well defined in $H^{s}(\mathbb{R}^{N})$ since $H^{s}(\mathbb{R}^{N})\hookrightarrow L^{r}(\mathbb{R}^{N})$ is continuous and locally compact for
$2\leq r< 2^{*}_{s}$ with $2^{*}_{s}=\frac{2N}{N-2s}$ if $N>2s$, and $2^{*}_{s}=+\infty$ if $N\leq 2s$, see \cite[Theorems 6.7,  6.9, Corollary 7.2]{DNE}. Moreover, $J_{a}(u)\in C^{1}(H^{s}(\mathbb{R}^{N}))$, and for any $\varphi\in H^{s}(\mathbb{R}^{N})$
\begin{equation*}
  \langle J^{'}_{a}(u), \varphi\rangle=2\int_{\mathbb{R}^{N}}\big[(-\Delta)^{\frac{s}{2}}u(-\Delta)^{\frac{s}{2}}\varphi+V(x)u\big]dx
  -2a\int_{\mathbb{R}^{N}}m(x)|u|^{\frac{4s}{N}}u\varphi dx.
\end{equation*}
Therefore, any minimizer of \eqref{eq1.4} is a least energy solution with prescribed $L^{2}$-norm of \eqref{eq1.1}. So, in what follows, we focus on the existence and asymptotic behaviors for the minimizers of the minimization problem \eqref{eq1.4}.

To state  our main results, we first recall some results on the following equation
\begin{equation}\label{eq1.5}
  (-\Delta)^{s}u+u=|u|^{\frac{4s}{N}} u,\quad x\in\mathbb{R}^{N}.
\end{equation}
For $N\geq 1$, Frank-Lenzmann-Silvester \cite{RF2} proved that \eqref{eq1.5} has a  unique and non-degeneracy radial ({up to translations}) ground state solutions $U(x)\in H^{s}(\mathbb{R}^{N})$. Moreover, the ground state solution is algebraic decay
\begin{equation}\label{eq1.6}
  \frac{C_{1}}{1+|x|^{N+2s}}\leq U(x)\leq\frac{C_{2}}{1+|x|^{N+2s}},
\end{equation}
 where $C_{1} \geq C_{2}>0$ are  constants depending only on $s$ and $N$. These are the generalizations on the results obtained in \cite{CA,RF1} for $N=1$. The same as Lemma 8.2 of \cite{RF2}, we know also that the ground state $U(x)$ of \eqref{eq1.5} satisfies the following Pohozaev identity:
\begin{equation}\label{eq1.7}
\int_{\mathbb{R}^{N}}|(-\Delta)^{\frac{s}{2}}U|^{2}dx
=\frac{N}{N+2s}\int_{\mathbb{R}^{N}}|U|^{\frac{4s}{N}+2}dx.
\end{equation}
Also, it is known (see, e.g., \cite{RF2,LW}) that $U(x)$ is the unique ({up to translations}) extremal of the following Gagliardo-Nirenberg type inequality
\begin{equation}\label{eq1.8}
  \|u\|_{\frac{4s}{N}+2}^{\frac{4s}{N}+2}\leq \frac{N+2s}{Na_s^{*}}\|(-\Delta)^{\frac{s}{2}}u\|_{2}^{2}\|u\|_{2}^{\frac{4s}{N}},
  \ \forall u\in H^{s}(\mathbb{R}^{N}),
\end{equation}
 where $a_s^{*}>0$ is given by
\begin{equation}\label{eq1.9}
  a_s^{*}=\|U(x)\|_{2}^{\frac{4s}{N}}.
\end{equation}
Moreover,  applying \eqref{eq1.8} for  $U(x)$, and combining with \eqref{eq1.7}, we see that
\begin{equation}\label{eq1.10}
  \int_{\mathbb{R}^{N}}|U|^{2}dx=\frac{2s}{N+2s}\int_{\mathbb{R}^{N}}|U|^{\frac{4s}{N}+2}dx.
\end{equation}

Throughout  this paper, we assume that $V(x)$ and $m(x)$ satisfy the following conditions.
\begin{itemize}
  \item [$(V_{1})$] $ V(x)\in L^{\infty}(\mathbb{R}^{N})$ and  the set $$\mathcal{Z}:=\big\{x_{0}\in\mathbb{R}^{N}: \lim_{x\rightarrow x_{0}}V(x)=V(x_{0})=\inf_{x\in\mathbb{R}^{N}}V(x)=0\big\}\not=\emptyset;$$

  \item [$(V_{2})$] $ 0<V^{\infty}:=\lim\limits_{|x|\rightarrow\infty}V(x)=\sup_{x\in\mathbb{R}^{N}}V(x)$ and {for some $\beta\in(0,2s),$
  $$V(x)-V^{\infty}<-\frac{1}{|x|^{\beta}} \ \mbox{for}\ |x|\gg 1. $$}

  \item [$(M_{1})$] $ 0<m^{\infty}:=\lim\limits_{|x|\rightarrow\infty}m(x)\leq m(x)\leq1$, and the set
   $$\mathcal{M}:=\big\{\bar{x}\in\mathbb{R}^{N}: m(\bar{x})=\sup_{x\in\mathbb{R}^{N}}m(x) \text{ and }\lim_{x\to\bar{x}}\frac{1-m(x)}{|x-\bar{x}|^{2s}}=0\big\}\not=\emptyset.$$

\end{itemize}

Our first result is about the existence of the minimizers for $I_{1}(a)$ in \eqref{eq1.4}.

\begin{theorem}\label{Th1.1}
 For $s \in (0,1)$ and $a_s^{*}>0$ given by \eqref{eq1.9},  let $(V_{1}),(V_{2})$ and $(M_{1})$ be satisfied. Then,
\begin{itemize}
  \item [\bf(i)] $I_{1}(a)$ has at least a nonnegative minimizer for any $a \in (0,a_s^{*})$;
  \item [\bf (ii)] $I_{1}(a)$ cannot be attained if $a>a_s^{*}$;
  \item [\bf (iii)] If $a=a_s^{*}$, then $I_{1}(a)$ cannot be attained provided $\mathcal{Z}\cap\mathcal{M}\not=\emptyset$.
\end{itemize}
 Moreover, $\lim\limits_{a\nearrow a_s^{*}}I_{1}(a)=I_{1}(a_s^{*})=0$, and
 $$I_{1}(a)>0 \text{ for } a \in (0,a_s^{*}), \  I_{1}(a)=-\infty \text{  for }  a>a_s^{*}.$$
\end{theorem}

We note that since $J_a(u)\geq J_a(|u|)$ for any $u\in H^s(\mathbb{R}^N)$, the {minimizers} of $I_{1}(a)$, if it exists, can always  be assumed to be nonnegative.\\

Our second theorem is about the concentration behavior for the minimizers of $I_{1}(a)$ under the  general bounded trapping type potential $V(x)$ as in $(V1)$ and $(V2)$.
\begin{theorem}\label{Th1.2}
Let $(V_{1}) (V_{2})$ and $(M_{1})$ be satisfied, and $\mathcal{Z}\cap\mathcal{M}\not=\emptyset$. If $u_{a}\geq 0$ is a minimizer of $I_{1}(a)$ for $a\in (0,a_s^{*})$. Then, for any sequence $\{a_{k}\} \subset (0,a_s^{*})$ with $a_{k}\nearrow a_s^{*}$ as $k\rightarrow \infty$, there exists a subsequence of $\{a_{k}\}$, still denoted by $\{a_{k}\}$, and    $\{\bar{z}_{k}\}\subset \mathbb{R}^N$ satisfying  $\lim\limits_{k\rightarrow \infty}\bar{z}_{k}=\bar{x}_{0}\in \mathcal{Z}\cap\mathcal{M}$, such that
\begin{equation}\label{eq1.11}
  \lim_{k\rightarrow \infty}\epsilon^{\frac{N}{2}}_{k}u_{a_{k}}(\epsilon_{k}x+\bar{z}_{k})= \frac{\big(\frac{2s}{N}\big)^{\frac{N}{4s}}
  U\Big(\big(\frac{2s}{N}\big)^{\frac{1}{2s}}(x+z_{0})\Big)}
  {\|U\|_{2}} \ \text{in}\ H^{s}(\mathbb{R}^{N}),
 \end{equation}
where $z_{0}\in \mathbb{R}^{N}$ is a fixed point, and $\epsilon_{k}>0$ is defined by
\begin{equation}\label{eq1.12}
  \epsilon_{k}:=\big(\int_{\mathbb{R}^{N}}|(-\Delta)^{\frac{s}{2}}u_{a_{k}}|^{2}dx\big)^{-\frac{1}{2s}}, \text{ and }  \epsilon_{k}
  \rightarrow 0 \ \text{as}\ k\rightarrow \infty.
\end{equation}
\end{theorem}

The above theorem shows that the minimizer $u_{a}$  of $I_{1}(a)$ must blow up around a point in $\mathcal{Z}\cap\mathcal{M}$. However, there is no any further information on the  blowup rate of $u_{a}$.  In order to see  more precise asymptotic behavior of the  minimizers, we require some further assumptions on $V(x)$ and $m(x)$ as follows.
\begin{itemize}
  \item [$(V_{3})$]  There exists  $x_{0}\in \mathbb{R}^N$ such that $\mathcal{Z}\cap\mathcal{M}=\{x_0\}$;
  \item [$(V_{4})$] $\lim\limits_{x\rightarrow x_{0}}$ $\frac{V(x)}{|x-x_{0}|^{p}}=C_{0}$, for some $C_{0}>0$ and  $p\in (0, N+4s)$;
  \item [$(M_{2})$] $ 0<m^{\infty}<m(x)\leq1$,  and
  $$\lim\limits_{x\rightarrow x_{0}}\frac{1-m(x)}{|x-x_{0}|^{q}}=\bar{C}, \text{ for some } \bar{C}>0 \text{ and } q\in(2s, N+8s+\frac{8s^{2}}{N}).$$
\end{itemize}

\begin{theorem}\label{Th1.3}
Let $(V_{1})-(V_{4})$ and  $(M_{2})$ be satisfied. If $\{a_{k}\}\subset(0,a_s^{*})$ is the convergent subsequence obtained in Theorem \ref{Th1.2} and $u_{a_{k}} $ is the corresponding minimizer of $I_{1}(a_{k})$.
Then,  for $C_0>0$ and $\bar{C}>0$ given in $(V_{4})$ and  $(M_{2})$, respectively, we have
\begin{enumerate}
  \item [\bf (i)]\begin{equation}\label{eq1.14}
  I_{1}(a_{k})\approx \frac{l+2s}{l}\big(\frac{l\gamma}{2s}\big)^{\frac{2s}{l+2s}}
  {a_s^{*}}^{-\frac{N+l}{l+2s}}\big(\frac{N+2s}{2s}(a_s^{*}-a_{k})\big)^{\frac{l}{l+2s}},
\end{equation}
where and in what follows, we mean $f_{k}\approx g_{k}$  that $f_k/g_k\rightarrow 1$ as $k\rightarrow \infty$, and $l=\min(q-2s,p)>0$,  $\gamma$ is given by
\begin{equation*}
  \gamma=
  \begin{cases}
  \big(\frac{N}{N+2s}\big)^{\frac{l+2s}{2s}}\bar{C}\Gamma_{1},\ \mbox{if}\ 0<q-2s<p;\\
  \\
  \big(\frac{N}{N+2s}\big)^{\frac{l}{2s}}C_{0}\Gamma_{2},
  \ \mbox{if}\ 0<p<q-2s;\\
  \\
  \big(\frac{N}{N+2s}\big)^{\frac{l+2s}{2s}}\bar{C}\Gamma_{1}
  +\big(\frac{N}{N+2s}\big)^{\frac{l}{2s}}C_{0}\Gamma_{2},\ \mbox{if}\ q-2s=p,
  \end{cases}
\end{equation*}
and
\begin{equation}\label{eq1.13}
\Gamma_1:=\int_{\mathbb{R}^{N}}|x|^{l+2s}|U|^{\frac{4s}{N}+2}dx, \  \Gamma_2:=\int_{\mathbb{R}^{N}}|x|^{l}|U|^{2}dx,
\end{equation}
with  $U(x)$ being the unique ground state solution of (\ref{eq1.2}).

  \item [\bf (ii)] $u_{a_{k}} $ satisfies \eqref{eq1.11}, and $\epsilon_{k}$ given by \eqref{eq1.12} can be precisely estimated as follows
\begin{equation}\label{eq1.15}
  \epsilon_{k}\approx
  \begin{cases}
  \big(\frac{N+2s}{l \bar{C}\Gamma_{1} }\big)^{\frac{1}{l+2s}} \big(\frac{2s}{N}\big)^{\frac{1}{2s}} \big(a^{*}_{s}\big)^{\frac{N-2s}{2s(l+2s)}}
    (a^{*}_{s}-a_{k})^{\frac{1}{l+2s}},
  \ \mbox{if}\ 0<q-2s<p;\\
  \\
  \big(\frac{N}{l C_{0}\Gamma_{2} }\big)^{\frac{1}{l+2s}} \big(\frac{2s}{N}\big)^{\frac{1}{2s}} \big(a^{*}_{s}\big)^{\frac{N-2s}{2s(l+2s)}}
    (a^{*}_{s}-a_{k})^{\frac{1}{l+2s}},
  \ \mbox{if}\ 0<p<q-2s;\\
  \\
  \big(\frac{l \bar{C}\Gamma_{1}}{N+2s}+\frac{l C_{0}\Gamma_{2}}{N}\big)^{-\frac{1}{l+2s}} \big(\frac{2s}{N}\big)^{\frac{1}{2s}} \big(a^{*}_{s}\big)^{\frac{N-2s}{2s(l+2s)}}
    (a^{*}_{s}-a_{k})^{\frac{1}{l+2s}},\ \mbox{if}\ q-2s=p.
  \end{cases}
\end{equation}
\end{enumerate}

\end{theorem}

\begin{remark}\label{Re1.1}
We mention that the  conditions
\begin{equation*}
   \int_{\mathbb{R}^{N}}|x|^{q}|U|^{\frac{4s}{N}+2}dx<+\infty, \text{ and }
   \int_{\mathbb{R}^{N}}|x|^{p}|U|^{2}dx<+\infty
\end{equation*}
are necessary in the above theorem and its proof. For this purpose, we have to make the  restriction on $p$ and $q$ such that $2s<q<N+8s+\frac{8s^{2}}{N}$ and $0<p<N+4s$ in Theorem \ref{Th1.3}, by considering {the decay rates} of $U(x)$ in \eqref{eq1.6}.
\end{remark}

This paper is organized as follows: In Section 2, Theorem $\ref{Th1.1}$ is proved based on  the concentration-compactness principle and some technical energy estimates. In Section 3, a general  discussion on the blowup behavior for the minimizers of $I_1(a)$, as $a\nearrow a^*$, is carried out, and  Theorem $\ref{Th1.2}$  is proved  under the conditions of $(V_1)$, $(V_2)$ and $(M_1)$.  In Section 4,  the precise asymptotic behavior on the minimizers for $I_1(a)$ is investigated under some further conditions such as  $(V_3)$, $(V_4)$ and $(M_2)$, and Theorem $\ref{Th1.3}$ is proved.

\section{Existence of minimizers}
\noindent

This section is devoted to the existence of minimizers for $I_1(a)$. For this purpose, we need some preliminary lemmas.
Similar to Lemma 2.1 of \cite{ZD}, we  have the following lemma.
\begin{lemma}\label{Le2.3}
Suppose $V(x)$ satisfies $(V_{1})$, and $m(x)$ satisfies $(M_{1})$, then for $0<a<a_s^{*}$, the map $\lambda\mapsto I_{\lambda}(a)$ is a continuous function on $\big(0,(\frac{a_s^{*}}{a})^{\frac{N}{2s}}\big)$.
\end{lemma}
\begin{proof}
 Choose $\lambda\in\big(0,(\frac{a_s^{*}}{a})^{\frac{N}{2s}}\big)$  and
  \begin{equation}\label{eq2.1}
 \{\lambda_{n}\}\subset\big(0,(\frac{a_s^{*}}{a})^{\frac{N}{2s}}\big) \ \text{ such that }  \ \lim_{n\to\infty}\lambda_{n}=\lambda.
  \end{equation}
  For any $\epsilon>0$, there exists  $u\in S_\lambda$  such that $J_{a}(u)\leq I_{\lambda}(a)+\epsilon$. Let $w_{n}=\sqrt{\frac{\lambda_{n}}{\lambda}}u$, then   $w_{n}\in S_{\lambda_{n}}$ and
\begin{eqnarray*}
  J_{a}(w_{n}) &=&   \frac{\lambda_{n}}{\lambda}\int_{\mathbb{R}^{N}}|(-\Delta)^{\frac{s}{2}}u|^{2}dx
  +\frac{\lambda_{n}}{\lambda}\int_{\mathbb{R}^{N}}V(x)|u|^{2}dx \\
   &&-a\frac{N}{N+2s}\big(\frac{\lambda_{n}}{\lambda}\big)^{\frac{2s}{N}+1}
   \int_{\mathbb{R}^{N}}m(x)|u|^{\frac{4s}{N}+2}dx \\
   &=& \frac{\lambda_{n}}{\lambda}J_{a}(u)
   +a\frac{N}{N+2s}\frac{\lambda_{n}}{\lambda}\Big(1-\big(\frac{\lambda_{n}}{\lambda}\big)^{\frac{2s}{N}}\Big)
   \int_{\mathbb{R}^{N}}m(x)|u|^{\frac{4s}{N}+2}dx.
\end{eqnarray*}
This gives that
\begin{equation*}
  \limsup_{n\rightarrow\infty}I_{\lambda_{n}}(a)\leq \lim_{n\rightarrow\infty}J_{a}(w_{n})= J_{a}(u)\leq I_{\lambda}(a)+\epsilon.
\end{equation*}
By the arbitrariness of $\epsilon$, we have
\begin{equation}\label{eq2.4}
  \limsup_{n\rightarrow\infty}I_{\lambda_{n}}(a)\leq I_{\lambda}(a).
\end{equation}

On the other hand, since $V(x)\geq0$, it is easy to know  from \eqref{eq1.8} that $I_{\lambda_{n}}(a)\geq0$  for all $n\in\mathbb{N}^+$. Therefore, there exists $\{v_{n}\}\subset H^{s}(\mathbb{R}^{N})$
such that  $\|v_{n}\|_{2}^{2}=\lambda_{n}$ and
\begin{equation}\label{eq2.3}
I_{\lambda_{n}}(a)\leq J_{a}(v_{n})\leq I_{\lambda_{n}}(a)+\frac{1}{n}.\end{equation}
 Applying  \eqref{eq1.8} again, we have
\begin{eqnarray}
    J_{a}(v_{n})
   &\geq& \big(1-\frac{a}{a_s^{*}}\lambda_{n}^{\frac{2s}{N}}\big)
  \int_{\mathbb{R}^{N}}|(-\Delta)^{\frac{s}{2}}v_{n}|^{2}dx\geq0. \label{eq2.6}
\end{eqnarray}
By \eqref{eq2.1},  we see that $\{v_{n}\}$ is uniformly bounded in $H^{s}(\mathbb{R}^{N})$.
Let $h_{n}=\sqrt{\frac{\lambda}{\lambda_{n}}}v_{n}$, then,
\begin{equation*}
  J_{a}(h_{n})=\frac{\lambda}{\lambda_{n}}J_{a}(v_{n})
  +a\frac{N}{N+2s}\frac{\lambda}{\lambda_{n}}\Big(1-\big(\frac{\lambda}{\lambda_{n}}\big)^{\frac{2s}{N}}\Big)
  \int_{\mathbb{R}^{N}}m(x)|v_{n}|^{\frac{4s}{N}+2}dx.
\end{equation*}
Consequently, there exist $M_{1},M_{2}>0$ such that
\begin{equation*}
\begin{split}
  |J_{a}(h_{n})-J_{a}(v_{n})|&\leq \big|\frac{\lambda}{\lambda_{n}}-1\big||J_{a}(v_{n})|
  +a\frac{N}{N+2s}\frac{\lambda}{\lambda_{n}}\big|1-\big(\frac{\lambda}{\lambda_{n}}\big)^{\frac{2s}{N}}\big|
  \int_{\mathbb{R}^{N}}m(x)|v_{n}|^{\frac{4s}{N}+2}dx\\
&\leq M_{1}\big|\frac{\lambda}{\lambda_{n}}-1\big|
  +M_{2}\big|\big(\frac{\lambda}{\lambda_{n}}\big)^{\frac{2s}{N}}-1\big|.
  \end{split}
\end{equation*}
Together with \eqref{eq2.3}, we have
\begin{eqnarray*}
  I_{\lambda_{n}}(a) &\geq& J_{a}(v_{n})-\frac{1}{n}\geq J_{a}(h_{n})-M_{1}\big|\frac{\lambda}{\lambda_{n}}-1\big|
  -M_{2}\big|\big(\frac{\lambda}{\lambda_{n}}\big)^{\frac{2s}{N}}-1\big|-\frac{1}{n} \\
   &\geq & I_{\lambda}(a)-M_{1}\big|\frac{\lambda}{\lambda_{n}}-1\big|
   -M_{2}\big|\big(\frac{\lambda}{\lambda_{n}}\big)^{\frac{2s}{N}}-1\big|-\frac{1}{n}.
\end{eqnarray*}
This yields that
\begin{equation*}
  \liminf_{n\rightarrow\infty}I_{\lambda_{n}}(a)\geq I_{\lambda}(a).
\end{equation*}
Combining with \eqref{eq2.4}, we have $\lim_{n\rightarrow\infty}I_{\lambda_{n}}(a)= I_{\lambda}(a)$. This completes the proof of Lemma \ref{Le2.3}.
\end{proof}\\

\begin{lemma}\label{Le2.4}
Under the assumptions of Theorem \ref{Th1.1}, we have
\begin{equation*}
  I_{\lambda}(a)<V^{\infty}\lambda \ \text{for any }\ { a\in(0,a^*_s)}\text{ and }\lambda\in(0,1].
\end{equation*}
\end{lemma}
\begin{proof}
 Let $$u_{\tau}=\frac{\sqrt{\lambda}\tau^{\frac{N}{2}}}{\|U\|_{2}}U(\tau x),$$ where $\tau>0$ is sufficiently small, and  $U(x)$ is the unique positive solution of \eqref{eq1.2}. Then,  $u_{\tau}\in S_\lambda$, and it follows from   \eqref{eq1.7}, \eqref{eq1.9} and \eqref{eq1.10} that
\begin{equation*}
  \int_{\mathbb{R}^{N}}|(-\Delta)^{\frac{s}{2}}u_{\tau}|^{2}dx=\frac{N\lambda}{2s}\tau^{2s}, \int_{\mathbb{R}^{N}}m(x)|u_{\tau}|^{\frac{4s}{N}+2}dx \geq \frac{N+2s}{2s}\frac{m^{\infty}\lambda^{\frac{2s}{N}+1}}{a_s^{*}}\tau^{2s}.
\end{equation*}
Moreover, we claim that
\begin{equation*}
  \int_{\mathbb{R}^{N}}V(x)|u_{\tau}|^{2}dx\leq
  -{ \frac{\lambda}{2}}{a_s^{*}}^{-\frac{N}{2s}}\tau^{\beta}\int_{\mathbb{R}^{N}}\frac{1}{|x|^{\beta}}|U(x)|^{2}dx
  +V^{\infty}\lambda.
\end{equation*}
Indeed, by $(V_{2})$, there exists  $R_{0}>0$large enough such that
\begin{equation*}\label{eq1}
  V(x)-V^{\infty}\leq -\frac{1}{|x|^{\beta}}, \ \text{if}\ |x|\geq R_{0}.
\end{equation*}
Then,
\begin{eqnarray*}
& &     \int_{\mathbb{R}^{N}}V(x)|u_{\tau}|^{2}dx 
     = \frac{\lambda}{\|U\|^{2}_{2}} \int_{\mathbb{R}^{N}}V(\frac{x}{\tau})|U(x)|^{2}dx\\
    & =& \frac{\lambda}{\|U\|^{2}_{2}} \int_{B(0,R_{0}\tau)}V(\frac{x}{\tau})|U(x)|^{2}dx
         +\frac{\lambda}{\|U\|^{2}_{2}} \int_{B(0,R_{0}\tau))^{c}}V(\frac{x}{\tau})|U(x)|^{2}dx\\
        & \leq&  \frac{\lambda V^{\infty}}{\|U\|^{2}_{2}}\int_{B(0,R_{0}\tau)}|U|^{2}dx +
        \frac{\lambda}{\|U\|^{2}_{2}}\int_{B(0,R_{0}\tau))^{c}}(V^{\infty}-\frac{\tau^{\beta}}{|x|^{\beta}})|U|^{2}dx\\
        &=& \lambda V^{\infty}-\lambda {a^{*}}^{-\frac{N}{2s}}_{s} \tau^{\beta}
      \int_{B(0,R_{0}\tau))^{c}}\frac{1}{|x|^{\beta}}|U|^{2}dx.
 \end{eqnarray*}
 Note that  $\beta\in (0,2s)$ with $s \in (0,1)$, $N\geq2$ and $U(x)$ is bounded, we know that
\begin{equation*}
  \int_{B(0,R_{0}\tau))}\frac{1}{|x|^{\beta}}|U|^{2}dx
  \leq
  C\int^{R_{0}\tau}_{0}r^{-\beta}r^{N-1}dr=\bar{C}\tau^{N-\beta}\rightarrow 0 \ \text{as}\ \tau\rightarrow 0,
\end{equation*}
 which then shows that
\begin{equation*}
  \int_{B(0,R_{0}\tau))^{c}}\frac{1}{|x|^{\beta}}|U|^{2}dx=\int_{\mathbb{R}^{N}}\frac{1}{|x|^{\beta}}|U|^{2}dx-
  \int_{B(0,R_{0}\tau))}\frac{1}{|x|^{\beta}}|U|^{2}dx
 \geq \frac{1}{2}\int_{\mathbb{R}^{N}}\frac{1}{|x|^{\beta}}|U|^{2}dx.
\end{equation*}
Therefore,
$$\int_{\mathbb{R}^{N}}V(x)|u_{\tau}|^{2}dx\leq
  -\frac{\lambda}{2}{a_s^{*}}^{-\frac{N}{2s}}\tau^{\beta}\int_{\mathbb{R}^{N}}\frac{1}{|x|^{\beta}}|U(x)|^{2}dx
  +V^{\infty}\lambda,$$
and our claim is proved.

Since $\beta\in (0, 2s)$, based on the above facts  we know  that
there exist  $C_{1},C_{2}>0$, such that
 $$J_{a}(u_{\tau})\leq C_{1}\tau^{2s}-C_{2}\tau^{\beta}+V^{\infty}\lambda<V^{\infty}\lambda, \text{ if  $\tau$ is sufficiently small}.$$
\end{proof}

\begin{lemma}\label{Le2.5}
For any fixed $a\in(0,a_s^{*})$ and $\lambda>0$, we have
\begin{equation}\label{eq2.5}
  I_{\theta\lambda_{1}}(a)<\theta I_{\lambda_{1}}(a)\ \text{ for all } \lambda_{1}\in(0,\lambda), \theta\in(1,\frac{\lambda}{\lambda_{1}}],
\end{equation}
and
\begin{equation}\label{eq2.06}
  I_{\lambda}(a)<I_{\lambda_{1}}(a)+I_{\lambda-\lambda_{1}}(a)\ \text{for all}\ \lambda_{1}\in(0,\lambda).
\end{equation}
\end{lemma}

\begin{proof}
Let $\{u_{n}\}$ be a minimizing sequence for {$I_{\lambda_{1}}(a)$}, we claim that
\begin{equation}\label{eq2.9}
  \liminf_{n\rightarrow\infty}\int_{\mathbb{R}^{N}}m(x)|u_{n}|^{\frac{4s}{N}+2}dx>0.
\end{equation}
Otherwise, there exists a subsequence of $\{u_{n}\}$, still denoted by $\{u_{n}\}$, such that
\begin{equation}\label{eq2.07}
\lim_{n\rightarrow\infty}\int_{\mathbb{R}^{N}}m(x)|u_{n}|^{\frac{4s}{N}+2}dx=0.
\end{equation}
By the assumption  $(V_2)$, we see that
 for any $\epsilon>0$, there exists  $R>0$ such that
\begin{equation*}
  |V(x)-V^{\infty}|\leq \frac{\epsilon}{\lambda_{1}}, \text{ for any }  |x|\geq R.
\end{equation*}
Applying the  H\"{o}lder inequality, we know that
\begin{eqnarray}
   &&\big|\lim_{n\rightarrow\infty}\int_{\mathbb{R}^{N}}(V(x)-V^{\infty})|u_{n}|^{2}dx\big| \nonumber \\
   &&\leq\lim_{n\rightarrow\infty}\int_{B_{R}(0)}|V(x)-V^{\infty}||u_{n}|^{2}dx
   +\lim_{n\rightarrow\infty}\int_{B_{R}^{c}(0)}|V(x)-V^{\infty}||u_{n}|^{2}dx \nonumber  \\
   &&\leq \lim_{n\rightarrow\infty}C\int_{B_{R}(0)}m(x)|u_{n}|^{\frac{4s}{N}+2}dx
   +\lim_{n\rightarrow\infty}\epsilon \int_{B_{R}^{c}(0)}|u_{n}|^{2}dx \label{eq2.9c}
   \leq\epsilon.
\end{eqnarray}
Together with \eqref{eq2.07}, we have
\begin{eqnarray*}
  I_{\lambda_{1}}(a)&=& \lim_{n\rightarrow\infty}\int_{\mathbb{R}^{N}}|(-\Delta)^{\frac{s}{2}}u_{n}|^{2}dx
                       +\lim_{n\rightarrow\infty}\int_{\mathbb{R}^{N}}(V(x)-V^{\infty})|u_{n}|^{2}dx
                       +V^{\infty}\lambda_{1}\\
  && -\lim_{n\rightarrow\infty}a\frac{N}{N+2s}\int_{\mathbb{R}^{N}}m(x)|u_{n}|^{\frac{4s}{N}+2}dx\\
   &\geq& V^{\infty}\lambda_{1}-\epsilon,\ \forall \epsilon>0.
\end{eqnarray*}
This contradicts Lemma \ref{Le2.4} by letting $\epsilon\to0$, and thus \eqref{eq2.9} is proved.

Set $v_{n}(x)=\theta^{\frac{1}{2}}u_{n}(x)$, then $v_n(x)\in S_{\theta\lambda_{1}}$, and it follows from  \eqref{eq2.9} that
\begin{eqnarray*}
  I_{\theta\lambda_{1}}(a) &\leq& \liminf_{n\rightarrow\infty}\Big(\theta\int_{\mathbb{R}^{N}}\big(|(-\Delta)^{\frac{s}{2}}u_{n}|^{2}
   +V(x)|u_{n}|^{2}\big)dx-\frac{\theta^{1+\frac{2s}{N}}a N}{N+2s}
   \int_{\mathbb{R}^{N}}m(x)|u_{n}|^{\frac{4s}{N}+2}dx\Big)\\
   &<& \theta\liminf_{n\rightarrow\infty}\Big(\int_{\mathbb{R}^{N}}\big(|(-\Delta)^{\frac{s}{2}}u_{n}|^{2}
   +V(x)|u_{n}|^{2}\big)dx-\frac{N a}{N+2s}\int_{\mathbb{R}^{N}}m(x)|u_{n}|^{\frac{4s}{N}+2}dx\Big) \\
   &=& \theta I_{\lambda_{1}}(a).
\end{eqnarray*}
Thus, \eqref{eq2.5} holds.  Moreover,  \eqref{eq2.06} can be obtained by
applying  \cite[Lemma II.1]{PL1}.

\end{proof}

The next lemma generalizes the concentration compactness principle of \cite{PL1} to the fractional Laplace operator. Its proof can be found in  \cite[Lemma 2.4]{BHF}, we omit it here.
\begin{lemma}\label{Le2.6}
\cite[Lemma 2.4]{BHF}
If $\{u_{n}\}\subset S_\lambda$ is a bounded sequence in $H^{s}(\mathbb{R}^{N})$ ($N\geq 2$), then there exists a subsequence (still denoted  by $\{u_{n}\}$) satisfying one of the following properties.
\begin{itemize}
  \item [$(i)$] Vanishing: $\lim_{n\rightarrow\infty} \sup_{y\in\mathbb{R}^{N}}\int_{B(y,R)}|u_{n}|^{2}dx=0$ for any $R>0$.
  \item [$(ii)$] Dichotomy: There exists $\beta\in(0,\lambda)$ such that, for any  $\epsilon>0$, there exist $ n_{0}\in\mathbb{N}$ and two bounded sequences in $H^{s}(\mathbb{R}^{N})$, denoted by $\{u_{n,1}\}$ and $\{u_{n,2}\}$ (both depending on $\epsilon$), such that for every $n\geq n_{0}$,
\begin{equation*}
  \big|\int_{\mathbb{R}^{N}}|u_{n,1}|^{2}dx-\beta\big|\le\epsilon,\quad \big|\int_{\mathbb{R}^{N}}|u_{n,2}|^{2}dx-(\lambda-\beta)\big|\le\epsilon,
\end{equation*}
\begin{equation*}
  \int_{\mathbb{R^{N}}}\big(|(-\Delta)^{\frac{s}{2}}u_{n}|^{2}-|(-\Delta)^{\frac{s}{2}}u_{n,1}|^{2}-
  |(-\Delta)^{\frac{s}{2}}u_{n,2}|^{2}\big)dx\geq -2\epsilon,
\end{equation*}
and
\begin{equation*}
  \|u_{n}-(u_{n,1}+u_{n,2})\|_{r}\leq 4\epsilon\ \text{for all}\ r\in[2,2_s^{*}).
\end{equation*}
Furthermore, there exist $\{y_{n}\}\subset \mathbb{R}^{N}$ and $\{R_{n}\}\subset (0,\infty)$ with $\lim_{n\rightarrow\infty}R_{n}=\infty$, such that
\begin{equation*}
  \begin{cases}
  u_{n,1}=u_{n},\quad \text{if}\ |x-y_{n}|\leq R_{0},\\
  u_{n,1}\leq |u_{n}|,\quad \text{if}\ R_{0}\leq|x-y_{n}|\leq 2R_{0},\\
  u_{n,1}=0,\quad \text{if}\ |x-y_{n}|\geq 2R_{0},
  \end{cases}
\end{equation*}
and
\begin{equation*}
  \begin{cases}
  u_{n,2}=0,\quad \text{if}\ |x-y_{n}|\leq R_{n},\\
  u_{n,2}\leq |u_{n}|,\quad \text{if}\ R_{n}\leq|x-y_{n}|\leq 2R_{n},\\
  u_{n,2}=u_{n},\quad \text{if}\ |x-y_{n}|\geq 2R_{n}.
  \end{cases}
\end{equation*}
  \item [$(iii)$] Compactness: There exists a sequence $\{y_{n}\}\subset\mathbb{R}^{N}$ such that for all $\epsilon>0$, there exists $R(\epsilon)>0$ such that
\begin{equation*}
  \int_{B(y_{n},R(\epsilon))}|u_{n}|^{2}dx\geq \lambda-\epsilon \ \text{for all}\ n\in \mathbb{N}.
\end{equation*}
\end{itemize}

\end{lemma}

Now, we are ready to prove Theorem 1.1.

{\noindent\textbf{Proof of Theorem 1.1.} }
\textbf{(i)} Using \eqref{eq1.8}, it is easy to see that
\begin{equation}\label{eq2.8}
  J_{a}(u)\geq\big(1-\frac{a}{a_s^{*}}\big)\|(-\Delta)^{\frac{s}{2}}u\|_{2}^{2}+
  \int_{\mathbb{R}^{N}}V(x)|u|^{2}dx, \forall\ u\in S_1.
\end{equation}
Let $\{u_{n}\}\subset S_{1}$ be a minimizing sequence of $I_1(a)$.
By \eqref{eq2.8},  $\{u_{n}\}$ is  bounded  in $H^{s}(\mathbb{R}^{N})$, passing to a subsequence if necessary, we may assume that
\begin{equation*}
  u_{n}\overset{n}\rightharpoonup u\ \text{weakly in}\ H^{s}(\mathbb{R}^{N}),
\end{equation*}
\begin{equation*}
  u_{n}\overset{n}\rightarrow u\ \text{strongly in}\ L^{r}_{loc}(\mathbb{R}^{N}),\ { r\in[2,2^{*}_{s})},
\end{equation*}
for some $u\in H^{s}(\mathbb{R}^{N})$. In what follows, we want to show that both {\it Vanishing} and {\it Dichotomy}  in Lemma \ref{Le2.6} do not occur.

$\textbf{Case 1.}$ If {\em Vanishing} occurs, then, passing to a subsequence if necessary, we have
\begin{equation*}
    \lim_{n\rightarrow\infty}\sup_{y\in\mathbb{R}^{N}}\int_{B(y,R)}|u_{n}|^{2}dx=0\quad \text{for all}\ R>0.
\end{equation*}
Then, the vanishing lemma  \cite[Lemma I.1]{PL2} implies that  $\|u_{n}\|_{r}\rightarrow 0\ \text{as}\ n\rightarrow\infty$ for any $2<r< 2_s^{*}.$
Hence,
\begin{equation*}
  \int_{\mathbb{R}^{N}}m(x)|u_{n}|^{\frac{4s}{N}+2}dx\rightarrow 0 \text{ as }\ n\rightarrow\infty.
\end{equation*}
Moreover, similar to the discussions in \eqref{eq2.9c}, we know that
%
\begin{equation*}
  \lim_{n\to\infty}\int_{\mathbb{R}^{N}}V(x)|u_{n}|^{2}dx=V^{\infty}.
\end{equation*}
Thus, $I_{1}(a)=\lim\limits_{n\to\infty}J_a(u_n)\geq V^{\infty}$, which  contradicts Lemma \ref{Le2.4}.

\textbf{Case 2.} If {\em Dichotomy} occurs, then, for any $\epsilon>0$, there exist $\xi\in(0,1)$, $n_{0}\in\mathbb{N}$ and two bounded sequences $\{u_{n,1}\}$ and $\{u_{n,2}\}$ in $H^{s}(\mathbb{R}^{N})$ (both depending on $\epsilon$), such that, for every $n\geq n_{0}$,
\begin{equation*}
  supp\ u_{n,1}\cap supp\ u_{n,2}=\emptyset,
\end{equation*}
\begin{equation*}
  \big|\int_{\mathbb{R^{N}}}|u_{n,1}|^{2}dx-\xi\big|<\epsilon,\quad \big|\int_{\mathbb{R^{N}}}|u_{n,2}|^{2}dx-(1-\xi)\big|<\epsilon,
\end{equation*}
\begin{equation*}
  \int_{\mathbb{R^{N}}}\big(|(-\Delta)^{\frac{s}{2}}u_{n}|^{2}-|(-\Delta)^{\frac{s}{2}}u_{n,1}|^{2}-
  |(-\Delta)^{\frac{s}{2}}u_{n,2}|^{2}\big)dx\geq -2\epsilon,
\end{equation*}
and
\begin{equation*}
  \|u_{n}-(u_{n,1}+u_{n,2})\|_{r}\leq 4\epsilon \quad \text{for all}\ 2\leq r<2_s^{*}.
\end{equation*}
Then,
\begin{eqnarray}
   &&J_{a}(u_{n})-J_{a}(u_{n,1})
   -J_{a}(u_{n,2})  \nonumber\\
   &&\geq-2\epsilon+\int_{\mathbb{R^{N}}}V(x)\big(|u_{n}|^{2}-|u_{n,1}
   +u_{n,2}|^{2}\big)dx \nonumber \\
   &\quad&\quad -\frac{Na}{N+2s}\int_{\mathbb{R^{N}}}m(x)\big(|u_{n}|^{\frac{4s}{N}+2}-|u_{n,1}
   +u_{n,2}|^{\frac{4s}{N}+2}\big)dx.\label{eq2.90}
\end{eqnarray}
Note that for any $a, b\in \mathbb{R}$, $p>1$, there exists $ C>0$ such that
$$\big||a|^{p}-|b|^{p}\big|\leq C|a-b|\big(|a|^{p-1}+|b|^{p-1}\big).$$
 Then, there exists $C>0$ independent of $n$, such that
\begin{equation*}
  \big||u_{n}|^{\frac{4s}{N}+2}-|u_{n,1}+u_{n,2}|^{\frac{4s}{N}+2}\big|\leq C|u_{n}-(u_{n,1}+u_{n,2})|\big(|u_{n}|^{\frac{4s}{N}+1}+|u_{n,1}+u_{n,2}|^{\frac{4s}{N}+1}\big).
\end{equation*}
Since $\{u_{n}\}$ is bounded in $H^{s}(\mathbb{R}^{N})$, by the H\"{o}lder inequality and the boundedness of $m(x)$ we know that
\begin{eqnarray*}
   &&\big|\frac{Na}{N+2s}\int_{\mathbb{R^{N}}}m(x)\big(|u_{n}|^{\frac{4s}{N}+2}
      -|u_{n,1}+u_{n,2}|^{\frac{4s}{N}+2}\big)dx \big|  \\
   &&\leq C_{1}\|u_{n}-(u_{n,1}+u_{n,2})\|_{\frac{4s}{N}+2}\leq 4 C_{1}\epsilon.
\end{eqnarray*}
Similarly,
\begin{equation*}
  \big|\int_{\mathbb{R^{N}}}V(x)\big(|u_{n}|^{2}-|u_{n,1}+u_{n,2}|^{2}\big)dx\big|\leq
  C_{2}\|u_{n}-(u_{n,1}+u_{n,2})\|_{2}\leq 4 C_{2}\epsilon.
\end{equation*}
Thus,
\begin{eqnarray*}
   &&-2\epsilon+\int_{\mathbb{R^{N}}}V(x)\big(|u_{n}|^{2}-|u_{n,1}+u_{n,2}|^{2}\big)\\
   &&-a\frac{N}{N+2s}\int_{\mathbb{R^{N}}}m(x)\big(|u_{n}|^{\frac{4s}{N}+2}
   -|u_{n,1}+u_{n,2}|^{\frac{4s}{N}+2}\big)\\
   &&\geq-2\epsilon-4\epsilon C_{1}-4\epsilon C_{2}=: -\delta(\epsilon)\rightarrow 0\ \text{as}\  \epsilon\rightarrow 0.
\end{eqnarray*}
So, it follws from \eqref{eq2.90} that
\begin{eqnarray}
  I_{1}(a) &=& \lim_{n\rightarrow\infty}J_{a}(u_{n})\geq
  \liminf_{n\rightarrow\infty}{J_{a}(u_{n,1})+\liminf_{n\rightarrow\infty}J_{a}(u_{n,2})}-\delta(\epsilon)\nonumber  \\
   &\geq&\lim_{n\rightarrow\infty}I_{\eta_{n,1}(\epsilon)}(a)+\lim_{n\rightarrow\infty}I_{\eta_{n,2}(\epsilon)}(a)
   -\delta(\epsilon),\label{eq2.b}
\end{eqnarray}
%
where
\begin{equation*}
  \eta_{n,i}(\epsilon)=\int_{\mathbb{R^{N}}}|u_{n,i}|^{2}dx,\ i=1,2.
\end{equation*}
Passing to a subsequence, we may assume that
\begin{equation*}
  |\eta_{n,1}(\epsilon)-\xi|\rightarrow 0\quad\text{and}\quad |\eta_{n,2}(\epsilon)-(1-\xi)|\rightarrow 0,
  \quad \text{as} \quad \epsilon\rightarrow 0.
\end{equation*}
By Lemma \ref{Le2.3}, $I_{\lambda}(a)$ is a continuous function with respect to $\lambda \in \big(0,(\frac{a_s^{*}}{a})^{\frac{N}{2s}}\big)$. Then,
letting $\epsilon\rightarrow 0$ in  \eqref{eq2.b}, we have
$$I_{1}(a)\geq I_{\xi}(a)+I_{1-\xi}(a),$$
which however contradicts Lemma \ref{Le2.5}. So dichotomy cannot occur, neither.

\textbf{Case 3.} By the discussions of the above two cases, Lemma \ref{Le2.6} implies that only  {\it Compactness}  occurs, that is, there exists a sequence $\{y_{n}\}\subset\mathbb{R^{N}}$ such that, for all $\epsilon>0$, there exists $R(\epsilon)>0$ such that
\begin{equation}\label{2.a}
  \int_{B(y_{n},R(\epsilon))}|u_{n}|^{2}dx\geq1-\epsilon, \text{ for all $n\in\mathbb{N}$}.
\end{equation}
We \textbf{claim} that $\{y_{n}\}$ is bounded. Otherwise, if $|y_{n}|\overset{n}\rightarrow\infty$, then,
\begin{equation*}
  \int_{\mathbb{R}^{N}}V(x)|u_{n}(x)|^{2}dx
  =\int_{\mathbb{R}^{N}}V(x+y_{n})|u_{n}(x+y_{n})|^{2}dx\overset{n}\rightarrow V^{\infty}.
\end{equation*}
Using \eqref{eq1.8}, it is easy to see that
\begin{equation*}
 \int_{\mathbb{R}^{N}}|(-\Delta)^{\frac{s}{2}}u_{n}|^{2}dx-
 \frac{aN}{N+2s}\int_{\mathbb{R}^{N}}m(x)|u_{n}|^{\frac{4s}{N}+2}dx\geq 0.
\end{equation*}
Thus,
\begin{equation*}
  I_{1}(a)=\lim_{n\rightarrow \infty}J_{a}(u_{n})\geq V^{\infty},
\end{equation*}
which contradicts Lemma \ref{Le2.4}, so $\{y_{n}\}$ is a bounded sequence in $\mathbb{R}^{N}$.

Set $R_{0}:=\sup_{n}|y_{n}|$, since $u_{n}\overset{n}\rightarrow u$ strongly in $L^{2}(0,R(\epsilon)+R_{0})$, it follows from \eqref{2.a} that
\begin{equation*}
  \|u\|_{2}^{2}\geq\int_{B(0,R(\epsilon)+R_{0})}|u|^{2}dx=\lim_{n\rightarrow \infty}\int_{B(0,R(\epsilon)+R_{0})}|u_{n}|^{2}dx
  \geq 1-\epsilon, \forall \ \epsilon>0.
\end{equation*}
This indicates that   $\|u\|_{2}^{2}=1$ and $u_{n}\overset{n}\rightarrow u$ strongly in $L^{2}(\mathbb{R}^{N})$.  Consequently,
\begin{equation}\label{eq2.10}
  u_{n}\overset{n}\rightarrow u\ \text{strongly in}\ L^{r}(\mathbb{R}^{N})\ (2\leq r< 2_s^{*}),
\end{equation}
and then,
$$I_1(a)=\lim_{n\to\infty} J_a(u_n)\geq J_a(u)\geq I_1(a).$$
Thus, $u$ is a minimizer of $I_{1}(a)$. Moreover, by noting \eqref{eq2.6}, we also see that $I_{1}(a)>0$.
\vskip.1truein

\textbf{(ii)} Now, we show that there is no minimizer for $I_{1}(a)$ for any $a>a_s^{*}$. \\
Taking $\bar{x}\in\mathcal{M}$ and let
\begin{equation}\label{eq2.11}
  u_{t}(x)=\frac{t^{\frac{N}{2}}}{\|U(x)\|_{2}}U\big(t(x-\bar{x})\big),\  t>0.
\end{equation}
Then, $u_{t}(x)\in S_1$, and it follows from \eqref{eq1.7}, \eqref{eq1.9} and \eqref{eq1.10} that
\begin{eqnarray}
  \nonumber &&\int_{\mathbb{R}^{N}}|(-\Delta)^{\frac{s}{2}}u_{t}(x)|^{2}
  -a\frac{N}{N+2s}\int_{\mathbb{R}^{N}}m(x)|u_{t}(x)|^{\frac{4s}{N}+2}dx \\
  \nonumber&&=\frac{t^{2s}}{\|U\|_{2}^{2}}\Big(\|(-\Delta)^{\frac{s}{2}}U(x)\|_{2}^{2}-
  \frac{a}{a_s^{*}}\frac{N}{N+2s}\int_{\mathbb{R}^{N}}
  m\big(\frac{x}{t}+\bar{x}\big)|U|^{\frac{4s}{N}+2}dx\Big) \\
  \nonumber&&=\frac{t^{2s}}{(a_s^{*})^{\frac{N}{2s}}}\frac{N}{N+2s}
  \Big[\|U\|_{\frac{4s}{N}+2}^{\frac{4s}{N}+2}-
  \frac{a}{a_s^{*}}\int_{\mathbb{R}^{N}}|U|^{\frac{4s}{N}+2}dx\Big.\\
  \nonumber &&\quad\Big.+\frac{a}{a_s^{*}}\int_{\mathbb{R}^{N}}
  \big(1-m(\frac{x}{t}+\bar{x})\big)|U|^{\frac{4s}{N}+2}dx\Big]\\
  \nonumber &&=t^{2s}\Big[\frac{N}{2s}(\frac{a_s^{*}-a}{a})+
   \frac{N}{N+2s}\frac{a}{(a_s^{*})^{\frac{N+2s}{2s}}}\int_{\mathbb{R}^{N}}
   \big(1-m(\frac{x}{t}+\bar{x})\big)|U|^{\frac{4s}{N}+2}dx\Big].\\
   &&\label{eq2.12}
\end{eqnarray}
Note that
\begin{equation}\label{eq2.13}
  \lim_{t\rightarrow \infty}\int_{\mathbb{R}^{N}}V(x)|u_{_{t}}|^{2}dx=\lim_{t\to\infty}\frac{1}{\|U\|_2^2}\int_{\mathbb{R}^{N}}V(x/t+\bar x)|U|^{2}dx=V(\bar{x}).
\end{equation}
Letting $t\rightarrow +\infty$, we deduce from  \eqref{eq2.12}, \eqref{eq2.13} and the conditions $(V_{1})$  $(M_{1})$ that,
\begin{equation*}
  I_{1}(a)=-\infty, \text{ if } a>a_s^{*},
\end{equation*}
and part (ii) is proved.
\vskip .1truein

\textbf{(iii)} Finally, we consider the case of $a=a_s^{*}$. \\
In this case, the additional {condition} $\mathcal{Z}\cap\mathcal{M}\not=\emptyset$ is required. Choosing $x_{0}\in\mathcal{Z}\cap\mathcal{M}$, then $V(x_{0})=0 \text{ and }1-m(x)=o(|x-x_0|^{2s})\text{ as } x\rightarrow x_0.$
 Let $u_{t}(x)$ be the trial function defined by  \eqref{eq2.11} with $\bar x= x_0$. Similar to \eqref{eq2.12}, we have
\begin{eqnarray}
  \nonumber &&\int_{\mathbb{R}^{N}}|(-\Delta)^{\frac{s}{2}}u_{t}(x)|^{2}dx
  -a_s^{*}\frac{N}{N+2s}\int_{\mathbb{R}^{N}}m(x)|u_{t}(x)|^{\frac{4s}{N}+2}dx  \\
  \nonumber &&=\frac{a_s^{*}t^{2s}}{\|U\|_{2}^{\frac{4s}{N}+2}}\frac{N}{N+2s}\int_{\mathbb{R}^{N}}
  \big(1-m(\frac{x}{t}+x_0)\big)|U|^{\frac{4s}{N}+2}dx  \\
   &&=(a_s^{*})^{-\frac{N}{2s}}\frac{N}{N+2s}\int_{\mathbb{R}^{N}}
   \frac{1-m(\frac{x}{t}+x_0)}{|\frac{x}{t}|^{2s}}|x|^{2s}|U|^{\frac{4s}{N}+2}dx \label{eq2.14}\\
  \nonumber && \rightarrow 0, \quad \text{as}\ t\rightarrow +\infty.
\end{eqnarray}
This estimate and \eqref{eq2.13} means that $I_{1}(a_s^{*})\leq V(x_{0})=0$. On the other hand, $I_{1}(a_s^{*})\geq 0$ by \eqref{eq2.8}, so  $I_{1}(a_s^{*})=0$.

Letting $a\rightarrow a_s^{*}$ and then $t\rightarrow +\infty$  in  \eqref{eq2.12} and \eqref{eq2.13}, we see that  $V(x_0)\leq\limsup\limits_{a\nearrow a_s^{*}}I_{1}(a)\leq V(x_{0})$, which gives $\lim\limits_{a\nearrow a_s^{*}}I_{1}(a)=I_{1}(a_s^{*})$.

Hence, it remains to show that $I_{1}(a_s^{*})$ has no minimizers. Otherwise, if $u(x)\in S_1$ is a  minimizer for $I_{1}(a_s^{*})=0$, then
\begin{equation}\label{eq2.15}
  \int_{\mathbb{R}^{N}}V(x)|u|^{2}dx=\inf_{x\in\mathbb{R}^{N}}V(x)=0,
\end{equation}
and by \eqref{eq1.8},
\begin{eqnarray*}
   &&\int_{\mathbb{R}^{N}}|(-\Delta)^{\frac{s}{2}}u|^{2}dx
     =\frac{a_s^{*}N}{N+2s}\int_{\mathbb{R}^{N}}m(x)|u|^{\frac{4s}{N}+2}dx  \\
   &&\leq\frac{a_s^{*}N}{N+2s}\int_{\mathbb{R}^{N}}|u|^{\frac{4s}{N}+2}dx\leq
     \int_{\mathbb{R}^{N}}|(-\Delta)^{\frac{s}{2}}u|^{2}dx.
\end{eqnarray*}
Therefore, using $I_{1}(a_s^{*})=0$, we have
\begin{equation*}
  \int_{\mathbb{R}^{N}}|(-\Delta)^{\frac{s}{2}}u|^{2}dx
  =\frac{a_s^{*}N}{N+2s}\int_{\mathbb{R}^{N}}|u|^{\frac{4s}{N}+2}dx.
\end{equation*}
Combining the property of \eqref{eq1.8}, this implies that $u$ has to be equal to $U$, up to a translation and dilation, which contradicts \eqref{eq2.15}. Thus, $I_1(a^*)$ cannot be attained.
{\hfill $\Box$}

\section{Asymptotic behavior under general bounded trapping potential}
\noindent

In this section, we aim to prove Theorem \ref{Th1.2}, which is to give a description on  the asymptotic behavior of the minimizers for $I_{1}(a)$ as $a\nearrow a_s^{*}$ under  general bounded trapping potentials.

Let $u_{a}$ be a minimizer of $I_{1}(a)$, then $u_{a}$ satisfies the Euler-Lagrange equation
\begin{equation}\label{3.1}
  (-\Delta)^{s}u_{a}(x)+V(x)u_{a}(x)=\frac{\lambda_{a}}{2}u_{a}+am(x)|u_{a}(x)|^\frac{4s}{N}u_{a}(x),
\end{equation}
where $\lambda_{a}\in \mathbb{R}$ is a suitable Lagrange multiplier. We give first the following lemma.

\begin{lemma}\label{le3.1}
Under the assumptions of Theorem \ref{Th1.2}, let $u_{a}\in S_{1}$ be a minimizer of $I_{1}(a)$ with $a\in (0,a_s^{*})$. Define
\begin{equation}\label{3.2}
  \epsilon^{-2s}_{a}:=\int_{\mathbb{R}^{N}}|(-\Delta)^{\frac{s}{2}}u_{a}|^{2}dx.
\end{equation}
Then
\begin{enumerate}
  \item [\bf{(i)}]
\begin{equation}\label{3.3}
  \epsilon_{a}\rightarrow 0,\ \text{as}\ a\nearrow a_s^{*}.
\end{equation}

  \item [\bf{(ii)}] There exists a sequence $\{\bar{z}_{a}\}\subset \mathbb{R}^{N}$ such that
\begin{equation}\label{3.4}
  \bar{w}_{a}(x):=\epsilon^{\frac{N}{2}}_{a}u_{a}(\epsilon_{a}x+\bar{z}_{a})\in S_{1}
\end{equation}
and
\begin{equation}\label{3.5}
  \int_{\mathbb{R}^{N}}|(-\Delta)^{\frac{s}{2}}\bar{w}_{a}|^{2}dx=1,\quad
\frac{a_s^{*}N}{N+2s}\int_{\mathbb{R}^{N}}m(\epsilon_{a}x+\bar{z}_{a})|\bar{w}_{a}|^{\frac{4s}{N}+2}dx
  \overset{a\nearrow a_s^{*}}\longrightarrow 1.
\end{equation}
Furthermore, there exist $R,\eta>0$ such that
\begin{equation}\label{3.6}
  \lim_{a\nearrow a_s^{*}}\int_{B_{R}(0)}|\bar{w}_{a}|^{2}dx\geq \eta>0.
\end{equation}
  \item [\bf{(iii)}] Up to a subsequence, there holds
\begin{equation}\label{3.7}
  \lim_{a\nearrow a_s^{*}}\bar{z}_{a}=x_{0}\in \mathcal{Z}.
\end{equation}
\end{enumerate}

\end{lemma}

\begin{proof}
\textbf{(i)} If (\ref{3.3}) is false, then there exists a sequence $\{a_{k}\}$ with $a_{k}\nearrow a_s^{*}$ as $k\rightarrow \infty$ such that $\{u_{a_{k}}\}$ is bounded in $H^{s}(\mathbb{R}^{N})$ and
\begin{eqnarray*}
  J_{a_s^{*}}(u_{a_{k}}) &=& J_{a_{k}}(u_{a_{k}})+(a_{k}-a_s^{*})\frac{N}{N+2s}\int_{\mathbb{R}^{N}}m(x)|u_{a_{k}}|^{\frac{4s}{N}+2}dx \\
   &=& I_{1}(a_{k})+O(a_s^{*}-a_{k})\rightarrow I_{1}(a_s^{*}).
\end{eqnarray*}
This means that $\{u_{a_{k}}\}$ is a minimizing sequence of $I_{1}(a_s^{*})$. Similar to the proof of Theorem \ref{Th1.1} (i), by applying
Lemma \ref{Le2.6}, we know that there exists a subsequence of $\{a_{k}\}$ (still denoted by $\{a_{k}\}$) and $u_{0}\in H^{s}(\mathbb{R}^{N})$ such that
\begin{equation*}
  u_{a_{k}}\overset{k}\rightharpoonup u_{0} \ \text{weakly in}\ H^{s}(\mathbb{R}^{N}) \ \text{and}\ u_{a_{k}}\overset{k}\rightarrow u_{0} \ \text{strongly in}\ L^{2}(\mathbb{R}^{N}).
\end{equation*}
Thus,
\begin{equation*}
  0=I_{1}(a_s^{*})\leq J_{a_s^{*}}(u_{0})\leq \lim_{k\rightarrow \infty}J_{a_{k}}(u_{a_{k}})=\lim_{k\rightarrow \infty}I_{1}(a_{k})=0.
\end{equation*}
This means that $u_{0}$ is a minimizer of $I_{1}(a_s^{*})$, which contradicts Theorem \ref{Th1.1} (iii). So, part (i) is proved.
\vskip.1truein
\textbf{(ii).} By Theorem \ref{Th1.1}, $I_{1}(a)\rightarrow I_{1}(a_s^{*})=0$, as $a\nearrow a_s^{*}$. Then,
\begin{eqnarray}
 \nonumber  0&\leq& \int_{\mathbb{R}^{N}}|(-\Delta)^{\frac{s}{2}}u_{a}|^{2}dx
   -\frac{Na}{N+2s}\int_{\mathbb{R}^{N}}m(x)|u_{a}|^{\frac{4s}{N}+2}dx\\
 \nonumber &=&\epsilon^{-2s}_{a}-\frac{Na}{N+2s}\int_{\mathbb{R}^{N}}m(x)|u_{a}|^{\frac{4s}{N}+2}dx \\
 \label{eq3.7'}  &\leq& I_{1}(a)\rightarrow 0, \ \text{as}\ a\nearrow a_s^{*}.
\end{eqnarray}
Set
\begin{equation}\label{eq3.7}
w_a(x):=\epsilon^{\frac{N}{2}}_{a}u_{a}(\epsilon_{a}x).
\end{equation}
Then, by (\ref{3.3}) and (\ref{eq3.7'})  we see that
\begin{equation}\label{3.8}
  \frac{a N}{N+2s}\int_{\mathbb{R}^{N}}m(\epsilon_a x)|w_{a}|^{\frac{4s}{N}+2}dx\rightarrow 1, \ \text{as}\ a\nearrow a_s^{*}.
\end{equation}
Now, we claim that there exist $\{z_a\}\subset \mathbb{R}^N$, $R>0$ such that
\begin{equation}\label{eq3.9}
  \lim_{a\nearrow a_s^{*}}\int_{B_{R}(z_a)}|{w}_{a}|^{2}dx\geq \eta>0.
\end{equation}
 Indeed,  if (\ref{eq3.9}) is false,  then by the  vanishing lemma  \cite[Lemma I.1]{PL2},  we know that $\|{w}_{a}(x)\|_{r}\rightarrow 0$ as $a\nearrow a_s^{*}$ for any $2<r < 2_s^{*}$, which contradicts \eqref{3.8}.
 Set
 $\bar w_a(x):=w_a(x+z_a),$
 it then follows from \eqref{eq3.7}-\eqref{eq3.9} by taking $\bar{z}_{a}=\epsilon_{a}z_{a}$ that (\ref{3.5}) and \eqref{3.6} hold.
%
\vskip.1truein

\textbf{(iii).} By \eqref{eq1.8} and {$I_{1}(a)=J_{a}(u_{a})$}, we know that
\begin{equation}\label{3.14}
  \int_{\mathbb{R}^{N}}V(x)|u_{a}(x)|^{2}\leq I_{1}(a)\rightarrow 0, \quad \text{as}\quad a\nearrow a_s^{*}.
\end{equation}
It then follows from (\ref{3.4})  that
\begin{equation}\label{3.15}
  \int_{\mathbb{R}^{N}}V(x)|u_{a}(x)|^{2}dx=\int_{\mathbb{R}^{N}}V(\epsilon_{a}x
  +\bar{z}_{a})|\bar{w}_{a}(x)|^{2}dx
    \rightarrow 0, \ \text{as} \ a\nearrow a_s^{*}.
\end{equation}
We first \textbf{claim} that $\{ \bar{z}_{a}\} $ is bounded in $\mathbb{R}^{N}$.
By contradiction,  if, passing to  a subsequence, $\bar{z}_{a} \stackrel{a}\rightarrow \infty$,  it follows from $\lim\limits_{|x|\rightarrow \infty} V(x)=\sup_{x\in\mathbb{R}^{N}}V(x)>0$  that there exists $C_{\delta}>0$ such that
\begin{equation*}
  \lim_{n\rightarrow\infty}V(\bar{z}_{a_{n}})\geq 2C_{\delta}>0,
\end{equation*}
this together with  Fatou's Lemma and (\ref{3.6}) imply that
\begin{equation*}
  \lim_{n\rightarrow \infty}\int_{\mathbb{R}^{N}}V(\epsilon_{a_{n}}x+\bar{z}_{a_{n}})|\bar{w}_{a_{n}}|^{2}dx
  \geq\int_{\mathbb{R}^{N}}\lim_{n\rightarrow\infty}V(\epsilon_{a_{n}}x+\bar{z}_{a_{n}})|\bar{w}_{a_{n}}|^{2}dx
  \geq C_{\delta}\eta,
\end{equation*}
which contradicts \eqref{3.15}.

Then,  we may assume that, for some $x_{0}\in \mathbb{R}^N$, $\bar{z}_{a_{n}} \stackrel{n} \rightarrow x_{0}$. If $x_{0}\not\in\mathcal{Z}$, as the above discussion for the boundedness of $\bar{z}_{a_{n}}$, we still get a contradiction with \eqref{3.15}. Therefore, (\ref{3.7}) is true and the proof is completed.
\end{proof}\\

{\noindent\textbf{Proof of Theorem \ref{Th1.2}.} }
For any sequence $\{a_{k}\}\subset (0,a_s^{*})$ with $a_k\nearrow a_s^{*}$ as $k\rightarrow\infty$,  let $u_{k}(x):=u_{a_{k}}(x)$ be the corresponding minimizers of $I_1(a_k)$,
we denote  $\epsilon_{a}, \ \bar{w}_{a}(x)$  and $\bar{z}_{a}$ defined in Lemma \ref{le3.1}, respectively,  by 
\begin{equation*}
  \epsilon_{k}:=\epsilon_{a_{k}}, \ \bar{w}_{k}(x):=\bar{w}_{a_{k}}(x)\geq 0,\text{ and }\bar{z}_{k}:=\bar{z}_{a_{k}}.
\end{equation*}
It follows from (\ref{3.5}) and (\ref{3.7}) that there exists a subsequence, still denote by $\{a_{k}\}$, such that
\begin{equation}\label{3.16}
  \lim_{k\rightarrow \infty}\bar{z}_{k}=\bar{x}_{0} \ \text{and}\ V(\bar{x}_{0})=0, \text{ i.e., } { \bar{x}_{0}\in \mathcal{Z}},
\end{equation}
and
\begin{equation*}
  \bar{w}_{k}(x)\rightharpoonup \bar{w}_{a_s^{*}}(x)\geq 0 \ \text{weakly in}\ H^{s}(\mathbb{R}^{N})
\end{equation*}
for some $\bar{w}_{a_s^{*}}(x)\in H^{s}(\mathbb{R}^{N})$. Moreover, \eqref{3.1}  implies that $\bar{w}_{k}(x)$ satisfies
\begin{eqnarray}
  \nonumber 2(-\Delta)^{s}\bar{w}_{k}(x)+2\epsilon^{2s}_{k}
  V(\epsilon_{k}x+\bar{z}_{a})\bar{w}_{k}(x)&=& 2am(\epsilon_{k}x
  +\bar{z}_{k})|\bar{w}_{k}(x)|^{\frac{4s}{N}}\bar{w}_{k}(x)\\
   && +\epsilon^{2s}\lambda_{k}\bar{w}_{k}(x). \label{3.17}
\end{eqnarray}
Using \eqref{3.1}, we know that
\begin{equation}\label{3.18}
  \lambda_{a}=2I_{1}(a)-a\frac{4s}{N+2s}\int_{\mathbb{R}^{N}}m(x)|u_{a}(x)|^{\frac{4s}{N}+2}dx.
\end{equation}
This and \eqref{3.8} implies
\begin{equation}\label{3.19}
  \epsilon_{a}^{2s}\lambda_{a}\rightarrow -\frac{4s}{N} \ \text{as}\ a\nearrow a_s^{*}.
\end{equation}
 Then, taking the weak limit in (\ref{3.17}), we know that $\bar{w}_{a_s^{*}}$ satisfies, in the weak sense,
\begin{equation}\label{3.20}
  (-\Delta)^{s}\bar{w}_{a_s^{*}}(x)
  =-\frac{2s}{N}\bar{w}_{a_s^{*}}(x)
  +a_s^{*}m(\bar{x}_{0})|\bar{w}_{a_s^{*}}(x)|^{\frac{4s}{N}}\bar{w}_{a_s^{*}}(x).
\end{equation}
Obviously, $\bar{w}_{a_s^{*}}\not\equiv0$ by \eqref{3.6}.
Moreover, it follows from \eqref{3.20} that $\bar{w}_{a_s^{*}}$ satisfies the following Pohozaev type identity (see e.g.,  \cite[Lemma 8.2]{RF2}),
\begin{equation}\label{3.21}
  \int_{\mathbb{R}^{N}}|\bar{w}_{a_s^{*}}|^{\frac{4s}{N}+2}dx= \frac{N+2s}{Na_s^{*}m(\bar{x}_{0})}
  \int_{\mathbb{R}^{N}}|(-\Delta)^{\frac{s}{2}}\bar{w}_{a_s^{*}}|^{2}dx.
\end{equation}
However,  the   Gagliardo-Nirenberg inequality \eqref{eq1.8} means that
\begin{equation}\label{3.22}
  \int_{\mathbb{R}^{N}}|\bar{w}_{a_s^{*}}|^{\frac{4s}{N}+2}dx\leq \frac{N+2s}{Na_s^{*}}
  \int_{\mathbb{R}^{N}}|(-\Delta)^{\frac{s}{2}}\bar{w}_{a_s^{*}}|^{2}dx
  \big(\int_{\mathbb{R}^{N}}|\bar{w}_{a_s^{*}}|^{2}dx\big)^{\frac{2s}{N}}.
\end{equation}
Thus, \eqref{3.21} and \eqref{3.22} imply that
\begin{equation*}
  1\leq \frac{1}{m(\bar{x}_{0})}\leq \big(\int_{\mathbb{R}^{N}}|\bar{w}_{a_s^{*}}|^{2}dx\big)^{\frac{2s}{N}} \leq 1.
\end{equation*}
Therefore,
\begin{equation}\label{3.23}
  m(\bar{x}_{0})=1 \quad \text{and} \quad \|\bar{w}_{a_s^{*}}\|_{2}=1.
\end{equation}
So,  $\bar{x}_{0}\in \mathcal{Z}\cap\mathcal{M}$.
Moreover, it follows from \eqref{3.23} and  \eqref{3.21} that $\bar{w}_{a_s^{*}}$ is an extremal of the  Gagliardo-Nirenberg inequality \eqref{eq1.8}, which together   with \eqref{3.20}  give that

\begin{equation}\label{3.24}
  \bar{w}_{a_s^{*}}(x)
  =\frac{(\frac{2s}{N})^{\frac{N}{4s}}
  U\big((\frac{2s}{N})^{\frac{1}{2s}}(x+z_{0})\big)}
  {\|U\|_{2}},
\end{equation}
where $z_{0}$ is a certain point in $\mathbb{R}^{N}$.  Since $\|\bar{w}_{a_s^{*}}\|^{2}_{2}=1$, by the norm preservation we further conclude that
\begin{equation*}
  \bar{w}_{k}\overset k\rightarrow \bar{w}_{a_s^{*}} \ \text{strongly in} \ L^{2}(\mathbb{R}^{N}).
\end{equation*}
Together with the boundness of $\bar{w}_{k}$ in $H^{s}(\mathbb{R}^{N})$, we have
\begin{equation}\label{eq3.25}
  \bar{w}_{k} \overset k\rightarrow \bar{w}_{a_s^{*}} \ \text{strongly in} \ L^{r}(\mathbb{R}^{N})\ \text{for}\
  2\leq r<2_s^{*}.
\end{equation}
 Furthermore, since  $\bar w_{k}$ and
$\bar{w}_{a_s^{*}}$ satisfy (\ref{3.17}) and (\ref{3.20}) respectively,
we know  that  $\bar w_{k}$ essentially converges to
$\bar{w}_{a_s^{*}}$ in $H^{s}(\mathbb{R}^{N})$, and we finish the proof of Theorem \ref{Th1.2}.
{\hfill $\Box$}\\

\section{Sharp energy estimates and precisely asymptotic behavior}
\noindent

In this section, by assuming that $V(x)$ is a polynomial type trapping potential, we first establish a sharp energy estimate for $I_{1}(a)$, upon which we then get a precisely asymptotic behavior for the minimizers of $I_{1}(a)$, and  Theorem \ref{Th1.3} is proved. Our first lemma is concerned with  the upper bound of $I_{1}(a)$, and the same lower bound is given in the proof of Theorem \ref{Th1.3}.
\begin{lemma}\label{Le4.1}
Under the assumptions of Theorem \ref{Th1.3} with  $l>0$ and $\gamma>0$ being defined there , then, for $k$ large enough,
\begin{equation}\label{eq4.1}
  0\leq I_{1}(a_{k})\leq \big(1+o(1)\big) \frac{l+2s}{l}\big(\frac{l\gamma}{2s}\big)^{\frac{2s}{l+2s}}
  {a_s^{*}}^{-\frac{N+l}{l+2s}}\big(\frac{N+2s}{2s}(a_s^{*}-a_{k})\big)^{\frac{l}{l+2s}},
\end{equation}
where and in what follows,  $o(1)$ denotes a quantity which goes to 0 as $k \rightarrow +\infty$.
\end{lemma}

\begin{proof}
Taking
\begin{equation}\label{eq4.2}
  u_{t}(x)=\frac{t^{\frac{N}{2}}}{\|U(x)\|_{2}}U\big(t(x-x_{0})\big)\in S_1,
\end{equation}
then, by the assumption $(V_4)$, we know that, for $t>0$ large enough,
\begin{eqnarray}\label{eq4.3}
  \nonumber\int_{\mathbb{R}^{N}}V(x)|u_{t}|^{2}dx&=&
  \frac{1}{\|U(x)\|_{2}^2}\int_{\mathbb{R}^{N}}V(x/t+x_0)|U|^{2}dx \\
  &=& \frac{C_{0}+o(1)}{(a_s^{*})^{\frac{N}{2s}}}t^{-p}\int_{\mathbb{R}^{N}}|x|^{p}|U(x)|^{2}dx,
\end{eqnarray}
where $\int_{\mathbb{R}^{N}}|x|^{p}|U(x)|^{2}dx<+\infty$ by Remark \ref{Re1.1}, and $o(1)$ denotes a quantity which goes to 0 as $t \rightarrow +\infty$.

Since \eqref{eq1.7}, \eqref{eq1.9} and \eqref{eq1.10}, we see that
\begin{eqnarray}\label{4.4}
   \nonumber&& \int_{\mathbb{R}^{N}}|(-\Delta)^{\frac{s}{2}}u_{t}(x)|^{2}dx-
   \frac{a N}{N+2s}\int_{\mathbb{R}^{N}}m(x)|u_{t}(x)|^{\frac{4s}{N}+2}dx  \\
   \nonumber&=& \big(\int_{\mathbb{R}^{N}}|(-\Delta)^{\frac{s}{2}}u_{t}(x)|^{2}dx-
\frac{a_s^{*}N}{N+2s}\int_{\mathbb{R}^{N}}|u_{t}(x)|^{\frac{4s}{N}+2}dx\big) \\
   \nonumber&&+(a_s^{*}-a)\frac{N}{N+2s}\int_{\mathbb{R}^{N}}|u_{t}(x)|^{\frac{4s}{N}+2}dx\\
   &&\nonumber+ \frac{a N}{N+2s}\int_{\mathbb{R}^{N}}\big[1-m(x)\big]|u_{t}(x)|^{\frac{4s}{N}+2}dx\\
   &=&\frac{N}{2s}\frac{a_s^{*}-a}{a_s^{*}}t^{2s}+\frac{a N}{N+2s}\int_{\mathbb{R}^{N}}\big[1-m(x)\big]|u_{t}(x)|^{\frac{4s}{N}+2}dx.
\end{eqnarray}
On the other hand, by the assumption of $(M_{2})$ we have
\begin{eqnarray}\label{4.6}
\nonumber && \frac{aN}{N+2s}\int_{\mathbb{R}^{N}}\big[1-m(x)\big]|u_{t}(x)|^{\frac{4s}{N}+2}dx \\
\nonumber  &&= a\frac{N}{N+2s}\|U\|_{2}^{-\frac{4s}{N}-2}t^{2s}
    \int_{\mathbb{R}^{N}}\big[1-m(\frac{x}{t}+x_{0})\big]|U(x)|^{\frac{4s}{N}+2}dx \\
  &&\leq\frac{1}{t^{q-2s}}\frac{N}{N+2s}\frac{\bar{C}+o(1)}{{a_s^{*}}^{\frac{N}{2s}}} \int_{\mathbb{R}^{N}}|x|^{q}|U|^{\frac{4s}{N}+2}dx,
\end{eqnarray}
where $\int_{\mathbb{R}^{N}}|x|^{q}|U|^{\frac{4s}{N}+2}dx<+\infty$ by Remark \ref{Re1.1}.
Hence, it follows from \eqref{4.4} and \eqref{4.6} that
\begin{eqnarray}\label{eq4.4}
   \nonumber&& \int_{\mathbb{R}^{N}}|(-\Delta)^{\frac{s}{2}}u_{t}(x)|^{2}dx
   -a\frac{N}{N+2s}\int_{\mathbb{R}^{N}}m(x)|u_{t}(x)|^{\frac{4s}{N}+2}dx  \\ \\
   \nonumber&\leq& t^{2s}\frac{N}{2s}\frac{a_s^{*}-a}{a_s^{*}}+
   \frac{1}{t^{q-2s}}\frac{N}{N+2s}\frac{\bar{C}+o(1)}{{a_s^{*}}^{\frac{N}{2s}}}
   \int_{\mathbb{R}^{N}}|x|^{q}|U|^{\frac{4s}{N}+2}dx,
\end{eqnarray}
where $o(1)$ is a quantity which goes to 0 as $k \rightarrow +\infty$

In order to get an optimal upper bound for $I_1(a_k)$, we need to consider the following three cases based on $l={\rm min}\{q-2s,p\}$, respectively.

\textbf{Case 1}: If $0<q-2s<p$, then $l=q-2s$. Using \eqref{eq4.3} and \eqref{eq4.4}, we see that, for any $t>0$ large enough,
\begin{equation}\label{eq4.5}
  I_{1}(a_k)\leq t^{2s}\frac{N}{2s}\frac{a_s^{*}-a_k}{a_s^{*}}
  +t^{-l}\frac{N}{N+2s}\frac{\bar{C}+o(1)}{{a_s^{*}}^{\frac{N}{2s}}}\Gamma_{1},
\end{equation}
where $\Gamma_1$ is defined by  \eqref{eq1.13}. Taking
\begin{equation}\label{eq4.6}
  t=\Big(\frac{l\bar{C}\Gamma_{1}}
  {(N+2s)(a_s^{*}-a_k){a_s^{*}}^{\frac{N-2s}{2s}}}\Big)^{\frac{1}{l+2s}},
\end{equation}
we know that \eqref{eq4.1} holds by \eqref{eq4.5} .

\textbf{Case 2}: If $0<p<q-2s$, then  $l=p$.  Based on \eqref{eq4.3} and \eqref{eq4.4}, we know
that, for any $t>0$ large enough,
\begin{equation*}
    I_{1}(a_k)\leq t^{2s}\frac{N}{2s}\frac{a_s^{*}-a_k}{a_s^{*}}
  +t^{-l}\frac{C_{0}+o(1)}{{a_s^{*}}^{\frac{N}{2s}}}\Gamma_{2},
\end{equation*}
where $\Gamma_2$ is defined by  \eqref{eq1.13}. Then  \eqref{eq4.1} can be obtained by taking
\begin{equation*}
  t=\Big(\frac{lC_{0}\Gamma_{2}}
  {N(a_s^{*}-a_k){a_s^{*}}^{\frac{N-2s}{2s}}}\Big)^{\frac{1}{l+2s}}.
\end{equation*}

\textbf{Case 3}: If $q-2s=p$, then  $l=p=q-2s$. By \eqref{eq4.3} and \eqref{eq4.4}, we have, for any $t>0$ large enough,
\begin{equation*}
   I_{1}(a_k)\leq t^{2s}\frac{N}{2s}\frac{a_s^{*}-a_k}{a_s^{*}}
   +t^{-l}\frac{1}{{a_s^{*}}^{\frac{N}{2s}}}\Big[\frac{N\big(\bar{C}+o(1)\big)}{N+2s}\Gamma_{1}
   +\big(C_{0}+o(1)\big)\Gamma_{2}\Big].
\end{equation*}
Then, \eqref{eq4.1} follows by taking
\begin{equation*}
  t=\Big(\frac{l\bar{C}\Gamma_{1}}{(N+2s)(a_s^{*}-a_k){a_s^{*}}^{\frac{N-2s}{2s}}}
  +\frac{lC_{0}\Gamma_{2}}{N(a_s^{*}-a_k){a_s^{*}}^{\frac{N-2s}{2s}}}\Big)
  ^{\frac{1}{l+2s}}.
\end{equation*}
This completes the proof.
\end{proof}\\

Finally, we turn to proving our main Theorem \ref{Th1.3}.\\

{\noindent\textbf{Proof of Theorem \ref{Th1.3}.} }
The same as in the  proof of Theorem \ref{Th1.2}, we denote
\begin{equation}\label{eq4.11}
  u_{k}:=u_{a_{k}}(x),\  \epsilon_{k}:=\epsilon_{a_{k}},\ \bar{z}_{k}:=\bar{z}_{a_{k}}\text{ and }\bar{w}_{k}(x):=\epsilon^{\frac{N}{2}}_{k}u_{k}(\epsilon_{k}x+\bar{z}_{k}),
\end{equation}
where $\epsilon_{a_{k}}$ and $\bar{z}_{a_{k}}$  are defined as Lemma \ref{le3.1}. Since \eqref{eq1.11}, we have
\begin{equation}\label{eq4.12}
  \bar{w}_{k}(x)\overset{k\rightarrow \infty}\longrightarrow\frac{\big(\frac{2s}{N}\big)^{\frac{N}{4s}}
  U\Big(\big(\frac{2s}{N}\big)^{\frac{1}{2s}}(x+z_{0})\Big)}
  {\|U\|_{2}} \ \text{a.e. in}\ x\in\mathbb{R}^{N}.
 \end{equation}
  It follows from  \eqref{eq1.9} and \eqref{eq4.11}   that
 \begin{eqnarray}\label{eq4.140}
  \nonumber I_{1}(a_{k})&=&\Big(
  \int_{\mathbb{R}^{N}}|(-\Delta)^{\frac{s}{2}}u_k|^{2}dx
  -\frac{a_s^{*}N}{N+2s}\int_{\mathbb{R}^{N}}|u_k|^{\frac{4s}{N}+2}dx\Big)+\int_{\mathbb{R}^{N}}V(x)|u_k|^{2}dx\nonumber\\
  &\ &+\frac{N}{N+2s}\Big[(a^*-a_{k})\int_{\mathbb{R}^{N}}|u_k|^{\frac{4s}{N}+2}dx+a_{k}
  \int_{\mathbb{R}^{N}}\big(1-m(x)\big)|u_k|^{\frac{4s}{N}+2}dx\Big]\nonumber\\
  \nonumber&\geq&\frac{a_s^{*}-a_{k}}{\epsilon^{2s}_{k}}\frac{N}{N+2s}
  \int_{\mathbb{R}^{N}}|\bar{w}_k|^{\frac{4s}{N}+2}dx
  +\int_{\mathbb{R}^{N}}V(\epsilon_{k}x+\bar{z}_k)|\bar{w}_k|^{2}dx  \\
   &\ &+\frac{a_{k}}{\epsilon^{2s}_{k}}\frac{N}{N+2s}
                 \int_{\mathbb{R}^{N}}\big[1-m(\epsilon_{k}x+\bar{z}_k)\big] |\bar{w}_k|^{\frac{4s}{N}+2}dx\nonumber\\
                 &=&\frac{a_s^{*}-a_{k}}{\epsilon^{2s}_{k}}\frac{N}{N+2s}
  \int_{\mathbb{R}^{N}}|\bar{w}_k|^{\frac{4s}{N}+2}dx+A_{k}\epsilon_{k}^{p}+\frac{Na_{k}}{N+2s}B_{k}\epsilon_{k}^{q-2s},
\end{eqnarray}
 where
 \begin{equation}\label{eq4.14}
  A_{k}:=\int_{\mathbb{R}^{N}}\frac{V(\epsilon_{k}x+\bar{z}_{k})}{|\epsilon_{k}x+\bar{z}_{k}-x_{0}|^{p}}
                \big|x+\frac{\bar{z}_{k}-x_{0}}{\epsilon_{k}}\big|^{p}|\bar{w}_{k}(x)|^{2}dx,
\end{equation}
 and
 \begin{equation}\label{eq4.16}
  B_{k}:=\int_{\mathbb{R}^{N}}\frac{1-m(\epsilon_{k}x+\bar{z}_{k})}{|\epsilon_{k}x+\bar{z}_{k}-x_{0}|^{q}}
          \big|x+\frac{\bar{z}_{k}-x_{0}}{\epsilon_{k}}\big|^{q}|\bar{w}_{k}(x)|^{\frac{4s}{N}+2}dx,
\end{equation}
$x_{0}$ is given by $(V_{3})$. Using  \eqref{eq1.9}, \eqref{eq1.10}, \eqref{3.24} and \eqref{eq3.25} we see that
\begin{eqnarray}\label{eq4.19}
   &\ &\lim_{k\to\infty}\int_{\mathbb{R}^{N}}|\bar{w}_{k}|^{\frac{4s}{N}+2}dx=
\frac{2s}{N}\|U\|_{2}^{-\frac{4s}{N}-2}\int_{\mathbb{R}^{N}}\big|
  U\big|^{\frac{4s}{N}+2}dx =\frac{N+2s}{a_s^{*}N}.
\end{eqnarray}
Therefore,  \eqref{eq4.140} and \eqref{eq4.19} give that
\begin{equation}\label{eq4.21}
  I_{1}(a_{k})\geq \frac{a_s^{*}-a_{k}}{a_s^{*}}\big(1+o(1)\big)\epsilon_{k}^{-2s}+A_{k}\epsilon_{k}^{p}+\frac{Na_{k}}{N+2s}B_{k}\epsilon_{k}^{q-2s},
\end{equation}
where $o(1)$ denotes a quantity which goes to 0 as $k \rightarrow +\infty$.

\vskip.1truein

We \textbf{claim} that
\begin{equation}\label{eq4.17}
  \lim_{k\rightarrow \infty}\big|\frac{\bar{z}_{k}-x_{0}}{\epsilon_{k}}\big|<\infty,
\end{equation}
i.e., $\big\{\frac{\bar{z}_{k}-x_{0}}{\epsilon_{k}}\big\}\subset \mathbb{R}^{N}$ is a bounded sequence. By contradiction, if \eqref{eq4.17} is false, then, $\lim\limits_{k\rightarrow \infty}\big|\frac{\bar{z}_{k}-x_{0}}{\epsilon_{k}}\big|=\infty$.
Since $\epsilon_{k}\overset{k}\rightarrow 0$ and $\bar{z}_{k}\overset{k}\rightarrow x_{0}$, it then follows from  $(V_{4})$ and $(M_{2})$ that
\begin{equation}\label{eq4.18}
  \lim_{k\rightarrow\infty} \frac{V(\epsilon_{k}x+\bar{z}_{k})}{|\epsilon_{k}x+\bar{z}_{k}-x_{0}|^{p}}=C_{0} \text{ and }
  \lim_{k\rightarrow\infty} \frac{1-m(\epsilon_{k}x+\bar{z}_{k})}{|\epsilon_{k}x+\bar{z}_{k}-x_{0}|^{q}}=\overline{C}\quad \text{for a.e.}\ x\in \mathbb{R}^{N}.
\end{equation}
 By Fatou's lemma,  \eqref{eq4.12}  and \eqref{eq4.18}  imply that
\begin{equation}\label{eq4.190}
  \lim_{k\rightarrow \infty} A_{k}= \lim_{k\rightarrow \infty}B_{k}=+\infty.
\end{equation}
Similar to the proof of the Lemma \ref{Le4.1}, it follows from \eqref{eq4.190} that
\begin{equation}\label{eq4.22}
  I_{1}(a_{k})\geq C(A_{k}, B_{k})(a_s^{*}-a_{k})^{\frac{l}{l+2s}},
\end{equation}
where $C(A_{k}, B_{k})>0$ depends on $A_{k}$ and $B_{k}$, and  $C(A_{k}, B_{k})\rightarrow\infty$ as $A_{k},\ B_{k}\rightarrow\infty$, which contradicts Lemma \ref{Le4.1}. Thus \eqref{eq4.17} is proved.

 As a consequence of   \eqref{eq4.17}, up to a subsequence, we may assume
\begin{equation}\label{eq4.23}
  \lim_{k\rightarrow \infty}\frac{\bar{z}_{k}-x_{0}}{\epsilon_{k}}=y_{0},\ \text{for some}\ y_{0}\in \mathbb{R}^{N}.
\end{equation}
Using \eqref{eq4.12}, \eqref{eq4.14}, \eqref{eq4.18} and Fatou's lemma, it is easy  to see that
\begin{eqnarray}\label{eq4.24}
\nonumber \lim_{k\rightarrow \infty}A_{k} &\geq& C_{0}\|U\|_{2}^{-2}\big(\frac{2s}{N}\big)^{\frac{N}{2s}}
                                      \int_{\mathbb{R}^{N}}|x+y_{0}|^{p}\big|U\big((\frac{2s}{N})^{\frac{1}{2s}}(x+z_{0})\big)\big|^{2}dx \\
\nonumber   &=&C_{0}\|U\|_{2}^{-2} \int_{\mathbb{R}^{N}} \big|(\frac{2s}{N})^{-\frac{1}{2s}}x-z_{0}+y_{0}\big|^{p}|U(x)|^{2}dx  \\
   &\geq& C_{0}\big(\frac{2s}{N}\big)^{-\frac{p}{2s}}(a_s^{*})^{-\frac{N}{2s}}\int_{\mathbb{R}^{N}}|x|^{p}|U(x)|^{2}dx,
\end{eqnarray}
where, the last inequality becomes equality if and only if $z_{0}=y_{0}$.

Similarly,
\begin{eqnarray}\label{eq4.25}
\nonumber \lim_{k\rightarrow \infty}B_{k} &\geq& \bar{C}\|U\|_{2}^{-\frac{4s}{N}-2}\big(\frac{2s}{N}\big)^{\frac{N}{2s}+1}
  \int_{\mathbb{R}^{N}}|x+y_{0}|^{q}\big|U\big((\frac{2s}{N})^{\frac{1}{2s}}(x+z_{0})\big)\big|^{\frac{4s}{N}+2}dx \\
   &\geq& \bar{C}\big(\frac{2s}{N}\big)^{-\frac{q-2s}{2s}}(a_s^{*})^{-\frac{N+2s}{2s}}
    \int_{\mathbb{R}^{N}}|x|^{q}|U(x)|^{\frac{4s}{N}+2}dx,
\end{eqnarray}
where, the last inequality becomes equality if and only if $z_{0}=y_{0}$.

Therefore, it follows from  \eqref{eq4.21}, \eqref{eq4.24} and \eqref{eq4.25} that,  for $k$ large enough,
\begin{eqnarray}\label{eq4.26}
\nonumber I_{1}(a_{k})&\geq&  \frac{a_s^{*}-a_{k}}{a_s^{*}}\big(1+o(1)\big)\epsilon_{k}^{-2s}+
                               C_{0}\big(1+o(1)\big)\big(\frac{2s}{N}\big)^{-\frac{p}{2s}}{a_s^{*}}^{-\frac{N}{2s}}
                               \epsilon_{k}^{p}\int_{\mathbb{R}^{N}}|x|^{p}|U(x)|^{2}dx \\
&&+\frac{N}{N+2s}\bar{C}\big(1+o(1)\big)\big(\frac{2s}{N}\big)^{-\frac{q-2s}{2s}}{a_s^{*}}^{-\frac{N}{2s}}\epsilon_{k}^{q-2s}
   \int_{\mathbb{R}^{N}}|x|^{q}|U(x)|^{\frac{4s}{N}+2}dx.
\end{eqnarray}
Since Lemma \ref{Le4.1},   to prove Theorem \ref{Th1.3} $\bf (i)$ we need to have a precise lower bound for $I_{1}(a_{k})$.  
Note that, in (\ref{eq4.26}),  $a_s^{*}-a_{k} \rightarrow 0$ and $\epsilon_{k} \rightarrow 0$, as $k\rightarrow \infty$. Then, in order to get a right lower bound for  (\ref{eq4.26}),  we should consider the following three cases based on  $l=\min\{q-2s,p\}$,
as in the proof of Lemma \ref{Le4.1}.

$\bullet$   $0<q-2s<p$,  that is,  $l=q-2s$:
In this case, for $k$ large enough, we may keep only  two  dominant  terms (i.e., the first  and the third terms) in \eqref{eq4.26}, then, by using the definitions in \eqref{eq1.13},  we have
\begin{eqnarray}\label{eq4.15a}
\nonumber I_{1}(a_{k})&\geq&  \frac{a_s^{*}-a_{k}}{a_s^{*}}\big(1+o(1)\big)\epsilon_{k}^{-2s}
       +\frac{N}{N+2s}\bar{C}\big(1+o(1)\big)\big(\frac{2s}{N}\big)^{-\frac{l}{2s}}{a_s^{*}}^{-\frac{N}{2s}}\Gamma_{1}\epsilon_{k}^{l}\\
&\geq&(1+o(1)) \frac{l+2s}{l}\big(\frac{l\gamma}{2s}\big)^{\frac{2s}{l+2s}}
  {a_s^{*}}^{-\frac{N+l}{l+2s}}\big(\frac{N+2s}{2s}(a_s^{*}-a_{k})\big)^{\frac{l}{l+2s}},
\end{eqnarray}
where $\gamma=\big(\frac{N}{N+2s}\big)^{\frac{l+2s}{2s}}\bar{C}\Gamma_{1}$, and  the last inequality becomes equality if and only if $\epsilon_{k}$ takes the value:
\begin{equation}\label{eq4.50a}
  \sigma_{k}:=\big(\frac{N+2s}{l \bar{C}\Gamma_{1} }\big)^{\frac{1}{l+2s}} \big(\frac{2s}{N}\big)^{\frac{1}{2s}} \big(a^{*}_{s}\big)^{\frac{N-2s}{2s(l+2s)}}
    (a^{*}_{s}-a_{k})^{\frac{1}{l+2s}}.
\end{equation}
Then, combinning Lemma \ref{Le4.1} , Theorem \ref{Th1.3} $\bf (i)$  is proved in this case.\\

$\bullet$ $0<p<q-2s$,  that is $l=p$:
In this case, for $k$ large enough, we may keep only  the first  and the second terms  in \eqref{eq4.26}, that is,
\begin{eqnarray}\label{eq4.15b}
\nonumber I_{1}(a_{k})&\geq&  \frac{a_s^{*}-a_{k}}{a_s^{*}}\big(1+o(1)\big)\epsilon_{k}^{-2s}
       +C_{0}\big(1+o(1)\big)\big(\frac{2s}{N}\big)^{-\frac{l}{2s}}{a_s^{*}}^{-\frac{N}{2s}}\Gamma_{2}\epsilon_{k}^{l}\\
&\geq&(1+o(1)) \frac{l+2s}{l}\big(\frac{l\gamma}{2s}\big)^{\frac{2s}{l+2s}}
  {a_s^{*}}^{-\frac{N+l}{l+2s}}\big(\frac{N+2s}{2s}(a_s^{*}-a_{k})\big)^{\frac{l}{l+2s}},
\end{eqnarray}
where $\gamma=\big(\frac{N}{N+2s}\big)^{\frac{l}{2s}}C_{0}\Gamma_{2}$, and  ``=" holds in the last inequality if and only if $\epsilon_{k}$ takes the value:
\begin{equation}\label{eq4.50b}
  \sigma_{k}:=\big(\frac{N}{l C_{0}\Gamma_{2} }\big)^{\frac{1}{l+2s}} \big(\frac{2s}{N}\big)^{\frac{1}{2s}} \big(a^{*}_{s}\big)^{\frac{N-2s}{2s(l+2s)}}
    (a^{*}_{s}-a_{k})^{\frac{1}{l+2s}},
\end{equation}
which then gives Theorem \ref{Th1.3} $\bf (i)$ by using Lemma \ref{Le4.1}.\\

  $\bullet$  $p=q-2s$, that is, $l=p=q-2s$:
   In this case, for $k$ large enough, we have,
\begin{eqnarray}\label{eq4.15c}
\nonumber I_{1}(a_{k})&\geq&  \frac{a_s^{*}-a_{k}}{a_s^{*}}\big(1+o(1)\big)\epsilon_{k}^{-2s}+
                               C_{0}\big(1+o(1)\big)\big(\frac{2s}{N}\big)^{-\frac{l}{2s}}{a_s^{*}}^{-\frac{N}{2s}}
                               \Gamma_{2}\epsilon_{k}^{l}\\
\nonumber &&+\frac{N}{N+2s}\bar{C}\big(1+o(1)\big)\big(\frac{2s}{N}\big)^{-\frac{l}{2s}}{a_s^{*}}^{-\frac{N}{2s}}\Gamma_{1}\epsilon_{k}^{l}\\
&\geq&(1+o(1)) \frac{l+2s}{l}\big(\frac{l\gamma}{2s}\big)^{\frac{2s}{l+2s}}
  {a_s^{*}}^{-\frac{N+l}{l+2s}}\big(\frac{N+2s}{2s}(a_s^{*}-a_{k})\big)^{\frac{l}{l+2s}},
\end{eqnarray}
where $ \gamma=\big(\frac{N}{N+2s}\big)^{\frac{l+2s}{2s}}\bar{C}\Gamma_{1}
  +\left(\frac{N}{N+2s}\right)^{\frac{l}{2s}}C_{0}\Gamma_{2}$, and ``=" holds in  the last inequality if and only if $\epsilon_{k}$ takes the value:
\begin{equation}\label{eq4.50c}
  \sigma_{k}:=\big(\frac{l \bar{C}\Gamma_{1}}{N+2s}+\frac{l C_{0}\Gamma_{2}}{N}\big)^{-\frac{1}{l+2s}} \big(\frac{2s}{N}\big)^{\frac{1}{2s}} \big(a^{*}_{s}\big)^{\frac{N-2s}{2s(l+2s)}}
    (a^{*}_{s}-a_{k})^{\frac{1}{l+2s}},
\end{equation}
which also gives  \eqref{eq1.14}  by Lemma \ref{Le4.1},  that is,  Theorem \ref{Th1.3} $\bf (i)$ is proved.\\

   In order to prove Theorem \ref{Th1.3} $\bf (ii)$, that is, $\epsilon_k>0$ satisfies  \eqref{eq1.15}, it is enough to show that
\begin{equation}\label{eq4.51}
  \lim_{k\rightarrow \infty} {\epsilon_{k}}/{\sigma_{k}}=1.
\end{equation}
   By the above discussions for $\sigma_{k}$ in \eqref{eq4.50a},  \eqref{eq4.50b} and \eqref{eq4.50c}, 
   we need to show \eqref{eq4.51} by three cases accordingly. Here, we consider only the case: $0<q-2s<p$,  i.e.,  $l=q-2s$. The other two cases can be done similarly.

Indeed, if \eqref{eq4.51} is false, then,  passing to a subsequence if necessary, we may assume that
\begin{equation*}
  \lim_{k\rightarrow \infty} {\epsilon_{k}}/{\sigma_{k}}=\theta\neq 1,\ \mbox{with}\ 0\leq\theta\leq\infty.
\end{equation*}
Using \eqref{eq4.15a}, for $k>0$ large enough,  we have
\begin{eqnarray*}
   I_{1}(a_{k})&\geq& \frac{a_s^{*}-a_{k}}{a_s^{*}}\big(1+o(1)\big)(\theta \sigma_{k})^{-2s}
       +\frac{N}{N+2s}\bar{C}\big(1+o(1)\big)\big(\frac{2s}{N}\big)^{-\frac{l}{2s}}{a_s^{*}}^{-\frac{N}{2s}}\Gamma_{1}(\theta \sigma_{k})^{l} \\
    &\geq& \big(1+o(1)\big)N\big(\frac{l \bar{C}\Gamma_{1}}{N+2s}\big)^{\frac{2s}{l+2s}}
           {a_s^{*}}^{-\frac{N+l}{l+2s}} (a_s^{*}-a_{k})^{\frac{l}{l+2s}}\big(\frac{1}{2s}\theta^{-2s}+\frac{1}{l}\theta^{l}\big)\\
    &>& (1+o(1)) \frac{l+2s}{l}\big(\frac{l\gamma}{2s}\big)^{\frac{2s}{l+2s}}
        {a_s^{*}}^{-\frac{N+l}{l+2s}}\big(\frac{N+2s}{2s}(a_s^{*}-a_{k})\big)^{\frac{l}{l+2s}}, \text{ if } \theta \neq 1,
\end{eqnarray*}
which contradicts Lemma \ref{Le4.1}, and \eqref{eq4.51} is true. Then,  we prove Theorem \ref{Th1.3} $\bf (ii)$. 

{\hfill $\Box$}\\

\end{document}